\documentclass[10pt]{amsart}
\usepackage{latexsym}
\usepackage{amsmath}
\usepackage{amssymb}
\usepackage{mathrsfs}
\usepackage{graphicx}
\usepackage{color}
\usepackage{pgfpages}
\usepackage{ifthen}
\usepackage{leftidx,tensor}
\usepackage[T1]{fontenc}
\usepackage[latin1]{inputenc}
\usepackage{mathtools}
\usepackage{comment}
\usepackage{dsfont}
\usepackage[nocompress]{cite}
\usepackage{amsaddr}

\usepackage[shortlabels]{enumitem}
\usepackage{aliascnt}
\usepackage[bookmarks=true,pdfstartview=FitH, pdfborder={0 0 0}, colorlinks=true,citecolor=red, linkcolor=blue]{hyperref}
\usepackage{bbm}
\usepackage{nicefrac}
\usepackage{tikz}
\usepackage{tikz-3dplot}
\usepackage{caption}
\theoremstyle{plain}
\newtheorem{thm}{Theorem}[section]
\newaliascnt{cor}{thm}
\newaliascnt{prop}{thm}
\newaliascnt{lem}{thm}
\newtheorem{cor}[cor]{Corollary}
\newtheorem{prop}[prop]{Proposition}
\newtheorem{lem}[lem]{Lemma}
\aliascntresetthe{cor}
\aliascntresetthe{prop}
\aliascntresetthe{lem} 
\theoremstyle{definition}
\newaliascnt{defn}{thm}
\newaliascnt{asu}{thm}
\newaliascnt{con}{thm}
\aliascntresetthe{defn}
\aliascntresetthe{asu}
\aliascntresetthe{con}
\newcounter{stp}
\newcounter{stpi}
\newcounter{stpci}
\newcounter{stpiii}

\theoremstyle{thm}
\newaliascnt{rem}{thm}
\newaliascnt{exa}{thm}
\newaliascnt{masu}{thm}
\newaliascnt{nota}{thm}
\newaliascnt{sett}{thm}
\newtheorem{rem}[rem]{Remark}
\newtheorem{exa}[exa]{Example}

\aliascntresetthe{rem}
\aliascntresetthe{exa}
\aliascntresetthe{masu}
\aliascntresetthe{nota}
\aliascntresetthe{sett}
\newtheorem{theoremA}{Theorem}

\numberwithin{equation}{section}

\setlist[enumerate]{font = \normalfont}

\newcommand {\N}	{\mathbb{N}}
\newcommand {\Z}	{\mathbb{Z}}

\newcommand {\R}	{\mathbb{R}}

\newcommand {\T}	{\mathbb{T}}

\renewcommand{\d}{\, \mathrm{d}}

\DeclareMathOperator{\divH}{div_{\H}}

\newcommand{\sE}{\mathcal{E}}

\renewcommand{\i}{\mathrm{i}}

\renewcommand{\H}{\mathrm{H}}

\newcommand{\sigmabar}{\bar{\sigma}}

\newcommand{\atm}{\mathrm{a}}
\newcommand{\ocn}{\mathrm{o}}

\newcommand{\air}{\mathrm{a}}

\newcommand{\ssatm}[1]{{#1}^{\atm}}
\newcommand{\ssocn}[1]{{#1}^{\ocn}}

	\newcommand{\Omegaatm}{\ssatm{\Omega}}
	\newcommand{\Omegaocn}{\ssocn{\Omega}}

	\newcommand{\Gau}{\Gamma_u^{\ocn}}
	\newcommand{\Gab}{\Gamma_b^{\ocn}}

	\newcommand{\dk}[1]{\partial_{#1}}
	\newcommand{\dt}{\dk{t}} 
	\newcommand{\dz}{\dk{z}} 
	
	\renewcommand{\phi}{\varphi}
	\newcommand{\vatm}{\ssatm{v}}
	\newcommand{\vocn}{v^{\ocn}}
	\newcommand{\wocn}{w^{\ocn}}

	\newcommand{\uocn}{u^{\ocn}}

	\newcommand{\hvair}{\hat{v}^{\air}}
	\newcommand{\hvocn}{\hat{v}^{\ocn}}
	
	\renewcommand{\bar}[1]{\overline{#1}}
	\newcommand{\vbar}{\bar{v}}
	
	\newcommand{\vtilde}{\tilde{v}}

	\newcommand{\nablaH}{\nabla_{\H}}
	\newcommand{\DeltaH}{\Delta_{\H}}

	\newcommand{\rC}{\mathrm{C}}
	\newcommand{\rL}{\mathrm{L}}

	\newcommand{\rH}{\H}
	
	\newcommand{\rA}{\mathrm{A}}
	
	\newcommand{\rE}{\mathrm{E}}

	\newcommand{\uair}{u^{\air}}
	\newcommand{\vair}{v^{\air}}

	\newcommand{\omegaair}{\omega^{\air}}

	\newcommand{\rhoair}{\rho^{\air}}
	\newcommand{\rhoocn}{\rho^{\ocn}}

	\newcommand{\vairbar}{\bar{v}^{\air}}
	\newcommand{\vocnbar}{\bar{v}^{\ocn}}

	\newcommand{\Omegaair}{\Omega^{\air}}
	\newcommand{\rX}{\mathrm X}
	\newcommand{\Gammaairu}{\Gamma_u^{\atm}}
	\newcommand{\Gammaairb}{\Gamma_b^{\atm}}
	
	\newcommand{\wair}{w^\atm}

	\renewcommand{\div}{\mathrm{div} \,}
	
	\title[Long-term behaviour of the CAO-system]{Long-term behaviour of the CAO-system: Optimal polynomial $\rH^1$-convergence to degenerate equilibria}
	\author{Tim Binz}
	\address{Technical University of Darmstadt\\
		        Department of Mathematics\\
		        Schlossgartenstr. 7\\
		        64289 Darmstadt\\
		        Germany}
	\email{binz@mathematik.tu-darmstadt.de}
	
	\subjclass{35Q86, 35Q35, 76D03, 35K55}%
	\keywords{CAO-problem, primitive equations, two-phase problems, nonlinear interface boundary conditions, nonlinear dynamics, long-term behaviour, center manifold}
\begin{document}

	\maketitle	
	
\begin{abstract}
This article is concerned with a definitive analysis of its long-term dynamics of the CAO-system. We establish the convergence of global strong solutions to constant equilibria in the full $\rH^1$-topology in the non-perturbative regime.
Remarkably, we demonstrate that the CAO-system exhibits a polynomial rate of decay and we determine the optimal rate. This phenomenon stands in stark contrast to the exponential convergence typically expected for uniformly parabolic systems on bounded domains. This algebraic slowing is shown to be a purely nonlinear effect, driven by the emergence of a slow manifold (center manifold) induced by nonlinear coupling conditions. To the best of our knowledge, this provides the first instance of a global polynomial decay result for a non-gradient, non-local dissipative system on a bounded domain.
\end{abstract}

\section{Introduction}
\label{sec:intro}

The \emph{primitive equations} constitute the standard model for large-scale oceanic and atmospheric dynamics, see \cite{Ped:87}.
Their origins trace back to the pioneering work of Richardson in 1922 \cite{Ric:22}.
They are derived from the \emph{Navier-Stokes equations} under the assumption of \emph{hydrostatic balance} in the pressure term.
The two-phase flow arising from the primitive equations is commonly referred to as the \emph{coupled atmosphere-ocean system} (CAO system), introduced by Lions, Temam, and Wang in \cite{LTW:93,LTW:95}.
For the velocity fields $\uair = (\vair,\wair)$ and $\uocn = (\vocn,\wocn)$, the system takes the form
\begin{equation}
	\left\{
	\begin{aligned} \dt \vair+ \vair \cdot \nablaH \vair + \wair \partial_p \vair  + \nablaH \Phi_s &= \Delta^{\atm} \vair, &&\text{ on } \Omegaair \times (0,T),  \\
		\div \uair &= 0, &&\text{ on } \Omegaair \times (0,T),\\
		\dt \vocn  + \vocn \cdot \nablaH \vocn + \wocn \dz \vocn + \nablaH \pi_s &= \Delta \vocn, &&\text{ on } \Omegaocn \times (0,T), \\
		\div \uocn &= 0, &&\text{ on } \Omegaocn \times (0,T), \\
	\end{aligned}
	\right. 
	\label{eq:CAO}
	\tag{CAO}
\end{equation}
where $\Delta^{\air} := \DeltaH + \partial_p (p^2 \partial_p )$. As a large-scale approximation of a two-phase Navier-Stokes system, the interface is effectively flat. Accordingly, we consider the following geometry: the atmosphere and the ocean are modeled as periodic layers $\Omegaair = \T^2 \times (p_a,p_s)$ with $0 < p_a < p_s$ and $\Omegaocn = \T^2 \times (-h,0)$ with $h > 0$. Their respective upper and lower boundaries are given by $\Gammaairb := \T^2 \times \{ p_a \}$, $\Gammaairu := \T^2 \times \{ p_s \}$, $\Gab := \T^2 \times \{ -h \}$, and $\Gau := \T^2 \times \{ 0 \}$. The fact that the interface is fixed determines the boundary condition for the vertical velocities
\begin{equation*}
		w^\air|_{\Gammaairb \cup \Gammaairu} = 0 \quad \text{ and } \quad \wocn|_{\Gab \cup \Gau} = 0 .
\end{equation*}

The central aspect of the two-phase problem is the coupling at the interface.\footnote{To be precise, in the CAO-system the two layers are formulated in different coordinate systems. The coupling occurs between the upper boundary of the atmospheric layer and the upper boundary of the oceanic layer. This reflects the fact that the atmospheric coordinates are based on the air pressure, which is proportional to height with a negative constant. In principle, one could perform a change of variables to obtain the coupling at a common interface; however, we retain the standard formulation to remain consistent with the literature. For a detailed discussion of the CAO-system, we refer to \cite{LTW:93,LTW:95} and the introduction of \cite{BBHZ:25}.} In the CAO-system, this coupling is described by the nonlinear \emph{wind-driven boundary conditions} or \emph{drag conditions}
\begin{equation} 
	\begin{aligned} 
		(\partial_p \vatm)|_{\Gammaairb} = 0 \quad &\text{ and } \quad (\partial_p \vatm)|_{\Gammaairu} = -p_s^{-1} \cdot |\vatm|_{\Gammaairu} -\vocn|_{\Gab}|\cdot(\vatm|_{\Gammaairu} -\vocn|_{\Gab} ) , \\ 
		(\partial_z \vocn)|_{\Gab} = 0 \quad &\text{ and } \quad (\partial_z \vocn)|_{\Gau} = p_s \cdot  |\vatm|_{\Gammaairu} -\vocn|_{\Gab}|\cdot  (\vatm|_{\Gammaairu} -\vocn|_{\Gab}) .
	\end{aligned} 
	\label{eq:wind bc}
\end{equation}

Lions, Temam and Wang proved in their pioneering work the existence of a weak solution to the primitive equations \cite{LTW:92a,LTW:92b} and to the CAO-system \cite{LTW:93,LTW:95}; its uniqueness
remains an open problem for both systems until today.

\medskip 

For the primitive equations with homogeneous Dirichlet and Neumann boundary conditions, global strong well-posedness in three space dimensions has been established for arbitrarily large initial data in $\rH^1$
in the celebrated work of Cao and Titi \cite{CT:07}. Related results can be found in \cite{KZ:07,Kob:07,HK:16,GGHHK:20,CLT:14a,CLT:15,GKVZ:14}. 
Recently, the global strong well-posedness result for the CAO-system was proved in \cite{BBHZ:25}.

\medskip 

Observing that two-phase flows with  linear coupling conditions are rather  well understood, we note that this is not the case for fully nonlinear interface conditions as formulated in CAO.

To avoid the difficulty of fully nonlinear interface conditions, they are often  replaced by simplified linear conditions as $\partial_z \vocn = |\vatm| \cdot \vatm$ for a given function $\vatm$.
As pointed out in \cite{LTW:93} and \cite{BS:01}, this modification is unjustified and generates unrealistic solutions from a physical point of view.
Moreover, this simplification overlooks the so-called \emph{eddy-killing effect}, a critical feedback mechanism where surface currents modulate wind stress. 
Implementing full mechanical air-sea coupling using relative velocities is a superior approach because it allows the stress to extract kinetic energy from eddies, thereby damping their variance and stabilizing western boundary currents, see for instance \cite{Ren:19}. 
Consequently, the simplified approach suppresses a vital pathway by which the atmosphere regulates mesoscale ocean energetics.

In contrast to the existing results in the literature, our aim is not to  simplify the interface condition,  but to consider the original CAO-system subject  to fully 
nonlinear interface conditions.
In this article, we show that the fully nonlinear interface conditions induce a slowing-down effect in the dynamics that is not captured by their simplified counterparts. This emphasizes the importance of the eddy-killing effect in the CAO-model. 

\smallskip 

For a discussion regarding the simplified linear boundary condition, we refer e.\ g.\ to the work of Bresch and Simon  \cite{BS:01}. For related results about the Navier-Stokes equations with wind-driven boundary conditions, we 
refer to the work of Desjardins and Grenier \cite{DG:00}, Bresch and Simon \cite{BS:01}, Bresch, Guill\'en-Gonz\'alez, Masmoudi and Rodr\'iguez-Bellido \cite{BGMR:03} and Dalibard and Saint-Raymond \cite{DS:09}. 
The primitive equations with stochastic wind-driven boundary conditions were studied in \cite{BHHS:24}, and with deterministic wind-driven boundary conditions in \cite{Bin:25}. In the latter work, a connection to the Ekman layer \cite{Ekm:05} is established, which appears natural in view of the fact that boundary layers typically arise due to rotation and wind stress.
For further background on boundary layers, we refer for example to we refer e.\ g.\ to the work of Dalibard and G\'erard-Varet \cite{DGV:17} or Dalibard and 
Saint-Raymond \cite{DS:09}. 
As explained in \cite{LTW:95}, the coupling conditions \eqref{eq:wind bc} incorporate such boundary-layer effects.
For further background on the CAO-system and primitive equations we refer to \cite{Ped:87,LTW:92a,LTW:92b,LTW:93,LTW:95,BBHZ:25,CDGG:06} and the references therein. 

\medskip 

\subsection{Main results}

\

The starting point of our investigation is the global well-posedness result 
for the CAO-system, recently been proven by the author together with F.~Brandt, M.~Hieber, and T.~Zöchling in \cite{BBHZ:25}. 
With the global well-posedness at hand the natural question arises how the CAO-system behaves in the long term $t \to + \infty$. 
The present work is primarily concerned with long-term behaviour of the CAO-system. 
This will yield to a complete description of the dynamics of the CAO-system. 
In the sequel we restrict ourself to the Hilbert space case $p = q = 2$ and recall \cite[Theorem 3.1]{BBHZ:25} in that case:

\begin{theoremA}[Global Strong Well-Posedness of the CAO-system]\label{thm:global} \mbox{} \\
	For every initial datum $(\vair_0, \vocn_0) \in \rH^1_{\sigmabar}(\Omegaatm) \times 
	\rH^1_{\sigmabar}(\Omegaocn)$ there is a unique, strong global solution $(\vair,w^a,\Phi)$, $(\vocn, \wocn, \pi)$\footnote{For the definitions of the spaces $\rH^1_{\sigmabar}(\Omegaatm)$ and $\rH^1_{\sigmabar}(\Omegaocn)$, as well as the manner in which $w^a$, $w^o$, $\Phi$, and $\pi$ are determined by $(\vair,\vocn)$, we refer the reader to \autoref{sec:prelim}.} to \eqref{eq:CAO}, subject to \eqref{eq:wind bc}, satisfying
	\begin{equation*}
		\begin{aligned}
			\vair &\in 
			\rH^{1}(0,T;\rL^2(\Omegaair;\R^2)) \cap \rL^2(0,T;\rH^{2}(\Omegaair;\R^2)) \cap \mathrm{C}(\R_+;\rH^1_{\sigmabar}(\Omegaatm)), \\
			\vocn &\in 
			\rH^{1}(0,T;\rL^2(\Omegaocn;\R^2)) \cap \rL^2(0,T;\rH^{2}(\Omegaocn;\R^2)) \cap \mathrm{C}(\R_+;\rH^1_{\sigmabar}(\Omegaocn)),
		\end{aligned}
	\end{equation*}
	for all $T > 0$. 
\end{theoremA}

Note that this statement implies that the CAO-system yields to a global semi-flow on the phase space $\rH^1_{\sigmabar}(\Omegaatm) \times 
\rH^1_{\sigmabar}(\Omegaocn)$.
With the global well-posedness at hand we can now aim for a deeper understanding of the dynamical behaviour of the CAO-system. 
Traditionally, the first step to do so is to identify conserved quantities and (strict) Lyapunov functionals of the system, the second step is to use them to identify the equilibria. 

\smallskip 

Two natural (physical) quantities of the CAO-system to consider are the \emph{total energy} 
\begin{equation*}
	\rE(t):= 
	\frac{1}{2}\int_{\Omegaatm} |\vair(t)|^2 + \frac{1}{2}\int_{\Omegaocn} |\vocn(t)|^2 
\end{equation*}
and the \emph{total momentum}
\begin{equation*}
	\mathrm{P}(t)
	:= \int_{\Omegaair} \vair(t) 
	+ \int_{\Omegaocn} \vocn(t) .
\end{equation*} 
Indeed, the CAO-system enjoys the following properties.    

\begin{theoremA}[Energy, Momentum and Equilibria]\label{thm:equilibria} \mbox{} \\
	The total energy $\rE$
	is a strict Lyapunov functional of \eqref{eq:CAO}, the total momentum $P$ 
	is a conserved quantity. 
	The critical points of the energy functional form the 
	set of equilibria. They are given by $\mathcal{E} = \{ (c,c) \colon c \in \R^2 \}$ and no equilibrium is normally stable. 
	More precisely, the CAO-system admits a four-dimensional slow manifold with a non-trivial reduced dynamics.  
\end{theoremA}

The fact that the energy is a strict Lyapunov functional shows that the CAO-system is dissipative. The conservation of the total momentum structures the phase space $\rH^1_{\sigmabar}(\Omegaatm) \times 
\rH^1_{\sigmabar}(\Omegaocn)$. This enables us to identify which equilibrium is approached for each initial datum. With this preparation we can now enter the next part of our studies: the convergence to equilibria. 
Despite the presence of the non-trivial dynamics on the slow manifold, the strict Lyapunov functional and point-wise in time a-priori estimates in the phase space $\rH^1_{\sigmabar}(\Omegaatm) \times 
\rH^1_{\sigmabar}(\Omegaocn)$ together with compactness of the orbits (in $\rH^s$ for $s <1$ due to compact Sobolev embeddings) already force the trajectories to converge (in $\rH^s$). They cannot oscillate on the center manifold. On the other hand, it follows by Angenent's parameter trick \cite{Ang:90a,Ang:90b} from the parabolicity of the system that the trajectories are analytic in time. Therefore, they cannot get stuck at an equilibrium in finite time. For more details see \autoref{cor:not finite time} and \autoref{cor:no recurrence}. 

\smallskip 

Our main result is a stronger convergence result: It does not only sharpens the convergence to the equilibria to the full $\rH^1$-norm of the phase-space but also determines the optimal convergence rate. That is where the non-normal stability of the equilibria and the non-trivial dynamics on the slow manifold enters the picture. 
Indeed, if the equilibria would be normally stable, the generalized principle of linearized stability \cite{PSZ:09} or \cite[Theorem 5.2.1]{PS:16} would imply \textit{exponential} convergence to equilibria close to it. Then, the $\rL^\infty_t\rH^2_{\bold{x}}$ a-priori estimates and the fact that the energy is a strict Lyapunov functional can be used to bridge from arbitrary initial data to such near the equilibrium. This strategy was carried out in \cite{Bin:25} for the primitive equations with inhomogeneous boundary conditions leading to exponential decay to equilibria in the $\rH^1$-norm.
The slow manifold might not be strong enough to stop the trajectories but it is able to slow them down from the exponential decay one expects for dissipative systems on bounded domains to a merely \textit{polynomial rate}.
Indeed, we obtain the following result.

\begin{theoremA}[Long-term behaviour of the CAO-system]\label{thm:long-term} \mbox{} \\
The global unique strong solution $(\vair,\vocn)$ from \autoref{thm:global} with initial data $(\vair_0,\vocn_0)$ converges in $\rH^1$ with polynomial rate 
	\begin{align*}
		\| \vair(t) - c \|_{\rL^2(\Omegaair)}
		+
		\| \vocn(t) - c \|_{\rL^2(\Omegaocn)}
		\leq C \cdot (1+t)^{-1} , \\
		\| \nabla^a \vair(t) \|_{\rL^2(\Omegaair)}
		+
		\| \nabla \vocn(t) \|_{\rL^2(\Omegaocn)}
		\leq C \cdot (1+t)^{-{\frac{3}{2}}}
	\end{align*}
	to the equilibrium $(c,c)$ from \autoref{thm:equilibria}, where
	\begin{equation*}
		c := \frac{\int_{\Omegaair} \vair_0 + \int_{\Omegaocn} \vocn_0}{|\Omegaair|+|\Omegaocn|} .
	\end{equation*}
	The convergence rates are optimal.  
	The convergence rates of the vertical velocities imply that the vertical velocities, the geopotential, and the oceanic pressure decay accordingly as
	\begin{align*}
		\| \wair(t) \|_{\rL^2(\Omegaair)}
		+
		\| \wocn(t) \|_{\rL^2(\Omegaocn)}
		\leq C \cdot (1+t)^{-{\frac{3}{2}}} , \\
		\| \nablaH \Phi(t) \|_{\rL^2(\Omegaair)}
		+
		\| \nablaH \pi(t) \|_{\rL^2(\Omegaocn)}
		\leq C \cdot (1+t)^{-2} .
	\end{align*}
\end{theoremA}

The reader may be surprised that the optimal decay rate is only polynomial. For dissipative, strongly parabolic systems posed on bounded domains, one typically expects exponential convergence to equilibrium, reflecting the presence of a spectral gap and the absence of degeneracies in the dynamics. The polynomial decay established here therefore indicates an unexpected degeneracy of the dynamics near the equilibria in phase space. This phenomenon is not caused by the equations themselves, which remain strongly parabolic and dissipative, but rather by the coupling conditions, whose structure becomes degenerate at the equilibria.
This result is in strong contrast to the situation of primitive equations where the system converges exponentially fast
see \cite{HK:16,Iid:25,Bin:25}.

\medskip 

Loosely speaking there are two dynamical effects which superimpose each other. Without the coupling by the wind-driven boundary conditions \eqref{eq:wind bc} the atmospheric and the oceanic velocities tend to their means at roughly exponential rate which follows from the asymptotics of primitive equations with Neumann boundary conditions. 
This would lead tp state $(c_1,c_2)$ where $c_1$ and $c_2$ in general do not coincide. It is easy to see that this cannot be a solution to \eqref{eq:wind bc}. Indeed, the boundary coupling \eqref{eq:wind bc} forces the two velocities to align. Due to the non-linear nature of the coupling this effect happens at merely polynomial scale as explained above. Therefore, it slows down the convergence from exponential to merely polynomial. 
The following graphics illustrates these two effects schematically: At time $t = 0$ the velocities are total chaotic. Very fast they tend to two (nearly) constant horizontal velocities with vertical velocities (close to) zero. Finally, the two velocities align at merely polynomial time scale. 

\medskip 

\tdplotsetmaincoords{70}{120}
\begin{figure}[h]
	\centering
	\makebox[\textwidth]{%
		\begin{tikzpicture}[tdplot_main_coords, scale=0.8]
			
			\newcommand{\boxes}{%
				
				\fill[blue!25, opacity=0.30]
				(0,0,0)--(4,0,0)--(4,4,0)--(0,4,0)--cycle;
				
				\fill[blue!20, opacity=0.18]
				(0,0,1)--(4,0,1)--(4,4,1)--(0,4,1)--cycle;
				
				\fill[blue!20, opacity=0.20]
				(0,0,0)--(4,0,0)--(4,0,1)--(0,0,1)--cycle;
				\fill[blue!20, opacity=0.20]
				(4,0,0)--(4,4,0)--(4,4,1)--(4,0,1)--cycle;
				\fill[blue!20, opacity=0.20]
				(4,4,0)--(0,4,0)--(0,4,1)--(4,4,1)--cycle;
				\fill[blue!20, opacity=0.20]
				(0,4,0)--(0,0,0)--(0,0,1)--(0,4,1)--cycle;
				
				\fill[gray!10, opacity=0.20]
				(0,0,1)--(4,0,1)--(4,4,1)--(0,4,1)--cycle;
				
				\fill[gray!5, opacity=0.10]
				(0,0,2)--(4,0,2)--(4,4,2)--(0,4,2)--cycle;
				
				\fill[gray!10, opacity=0.15]
				(0,0,1)--(4,0,1)--(4,0,2)--(0,0,2)--cycle;
				\fill[gray!10, opacity=0.15]
				(4,0,1)--(4,4,1)--(4,4,2)--(4,0,2)--cycle;
				\fill[gray!10, opacity=0.15]
				(4,4,1)--(0,4,1)--(0,4,2)--(4,4,2)--cycle;
				\fill[gray!10, opacity=0.15]
				(0,4,1)--(0,0,1)--(0,0,2)--(0,4,2)--cycle;
				
				\draw (0,0,0)--(4,0,0)--(4,4,0)--(0,4,0)--cycle;
				\draw (0,0,2)--(4,0,2)--(4,4,2)--(0,4,2)--cycle;
				
				\draw (0,0,0)--(0,0,2);
				\draw (4,0,0)--(4,0,2);
				\draw (4,4,0)--(4,4,2);
				\draw (0,4,0)--(0,4,2);
				
				\fill[blue!40, opacity=0.45]
				(0,0,1)--(4,0,1)--(4,4,1)--(0,4,1)--cycle;
				
				\draw[line width=0.6pt]
				(0,0,1)--(4,0,1)--(4,4,1)--(0,4,1)--cycle;
				
			}
			
			\foreach \x in {0.8,2.0,3.2}{
				\foreach \y in {0.8,2.0,3.2}{
					\begin{scope}
						\draw[->, black] (\x,\y,0.3) --
						++({0.28*cos(30*\x+10*\y)},
						{0.28*sin(25*\x)},
						{0.18*sin(15*\y)});
					\end{scope}
			}}
			
			\begin{scope}[xshift=0cm]
				\node at (2,0,-1.5) {$t=0$};
				\boxes
				
				\foreach \x in {0.8,2.0,3.2}{
					\foreach \y in {0.8,2.0,3.2}{
						\draw[->, black] (\x,\y,1.55) --
						++({0.38*sin(20*\x)},
						{0.38*cos(18*\y)},
						{0.25*cos(12*\x)});
				}}
				
				\node[anchor=east] at (5.2,7.5,1.5) {Ocean};
				\node[anchor=east] at (5.2,7.5,5) {Atmosphere};
				
			\end{scope}
			
			\begin{scope}[xshift=7cm]
				\node at (2,0,-1.5) {$t \gg 1$};
				\boxes
				
				\foreach \x in {0.8,2.0,3.2}{
					\foreach \y in {0.8,2.0,3.2}{
						\draw[->, black] (\x,\y,0.3) -- ++(0.35,0.05,0);
						\draw[->, black] (\x,\y,1.55) -- ++(-0.22,0.05,0);
				}}
				
				\node[anchor=east] at (5.2,7.5,1.5) {Ocean};
				\node[anchor=east] at (5.2,7.5,5) {Atmosphere};
				
			\end{scope}
			
			\begin{scope}[xshift=14cm]
				\node at (2,0,-1.5) {$t \to \infty$};
				\boxes
				
				\foreach \x in {0.8,2.0,3.2}{
					\foreach \y in {0.8,2.0,3.2}{
						\draw[->, black] (\x,\y,0.3) -- ++(0.18,0,0);
						\draw[->, black] (\x,\y,1.55) -- ++(0.22,0,0);
				}}
				
				\node[anchor=east] at (5.2,7.5,1.5) {Ocean};
				\node[anchor=east] at (5.2,7.5,5) {Atmosphere};
				
			\end{scope}
			
		\end{tikzpicture}
	}
	\caption{Evolution of the CAO-system in time}
	\label{fig:ocean-atmosphere}
\end{figure}
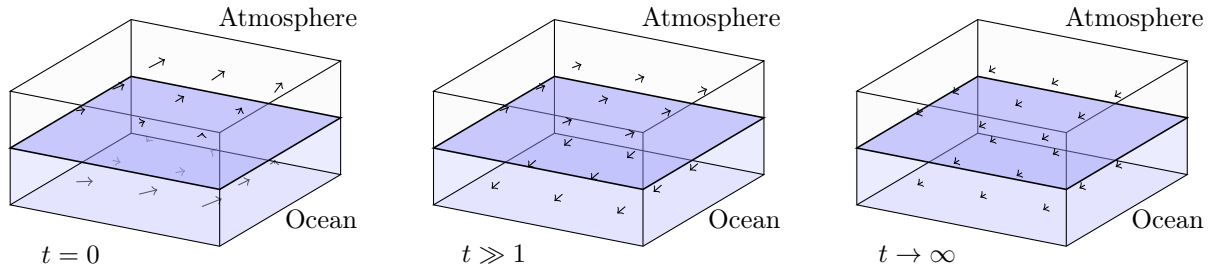

\medskip 

Let us to point out, that the \eqref{eq:CAO} system cannot be decoupled into two primitive equations and a simpler equation with a non-linear boundary conditions due to its non-local character. More precisely, the condition $\vair_0|_{\Gammaairb} = \vocn_0|_{\Gau}$, which would decouple the system, is not preserved along the semi-flow. 
This in particular, shows that in general for initial data $\vair_0|_{\Gammaairb} = \vocn_0|_{\Gau}$ the solution to \eqref{eq:CAO} does not coincide with the solutions to primitive equations with Neumann boundary conditions and one does not know whether they decay exponentially fast or not. 
Nevertheless, there are special initial data where the system decouples, see \autoref{exa:exponential decay}, and therefore decays exponentially fast. 

\subsection{Main challenges, new contributions and open problems}
\label{ssec:main challenges}

\

There are two powerful and far-reaching abstract frameworks for analyzing the long-term behaviour of parabolic evolution equations: the \emph{Lojasiewicz--Simon inequality} \cite{PW:11, Loj:62,Loj:65,Loj:84,Sim:83,Chi:03,CM:16,FS:00,Hua:06} and the \emph{(generalized) principle of linearized stability} \cite{Lun:95,PS:16,PSZ:08,PSZ:09,DPL:92,Hen:81}. Each has its own strengths and limitations.

\smallskip 

The Lojasiewicz--Simon inequality is well suited to systems whose dynamics become degenerate near equilibria. By allowing a non-standard Lojasiewicz exponent $\theta \neq \frac12$, it can capture degenerate dynamics near equilibria and thus accommodate polynomial decay rates. However, its applicability is restricted to gradient systems. In contrast, the (generalized) principle of linearized stability is not limited to gradient systems and applies equally well to equations with transport or other non-variational effects. Its applicability, however, requires the equilibria to be normally stable, in which case it yields exponential convergence to equilibrium.

\medskip

The \ref{eq:CAO}-system lies precisely outside the scope of both theories. On the one hand, the presence of the advection and pressure terms destroys the gradient structure, rendering the Lojasiewicz--Simon approach inapplicable. On the other hand, the coupling conditions become degenerate at the equilibria, so that the equilibria fail to be normally stable. Consequently, the system possess a center-manifold with a non-trivial dynamics on it and therefore the generalized principle of linearized stability cannot be applied either.

Furthermore, it is worth pointing out that the nonlinear coupling conditions are only of class $\mathcal{C}^{1,1}$, so the corresponding center manifold is, in general, expected to possess only $\mathcal{C}^1$ regularity. Consequently, the reduced dynamics cannot be obtained by performing a Taylor expansion of the center-manifold slaving function. 

\medskip

To the best of our knowledge, no previous example has been reported in which polynomial convergence to equilibrium for a dissipative, strongly parabolic system on a bounded domain is caused solely by the degeneration of the coupling conditions.
The proof developed in this work is robust and exploits the structural features responsible for the degeneracy. We therefore expect that the approach can be adapted to other evolution problems in which the degeneracy originates from boundary or coupling conditions rather than from the differential equations themselves.

\medskip 

We conclude this subsection by highlighting two open problems. 

First, a rigorous proof of the refined asymptotic behaviour described at the end of last subsection remains open. In particular, showing that the difference between the velocities and their means decays faster than the polynomial rate from \autoref{thm:long-term}. Although leveraging the derived equations and the center manifold description offers a plausible approach, the non-linear and non-local dynamics make this a highly non-trivial task. 
Because the central aim of the present work is to establish polynomial convergence rates, we defer a rigorous treatment of this faster decay to future research.

A second compelling question is whether our main result, \autoref{thm:long-term}, can be extended to the full $(p,q)$-scale. The global well-posedness result in \cite{BBHZ:25} holds in the $\rL^p_t \rL^q_x$-setting, giving rise to a continuous semi-flow on appropriate Besov spaces. However, for sufficiently large $p$ and $q$, the coupling conditions inherently modify the phase space, transforming it from a flat Banach space into a curved Banach manifold. 
Since our current methodology relies heavily on precise energy estimates, it cannot be readily adapted to this broader $(p,q)$-framework. 
Therefore, the study of the CAO-system on the full $(p,q)$-scale requires new analytic tools and we leave it for future studies.

\subsection{A toy example and the main ideas of the proof}

\

Let us now explain where the polynomial decay of the CAO-system comes from.
It arises due to the degeneracy of the coupling conditions \eqref{eq:wind bc}.
Indeed, the CAO-system consists of two incompressible primitive equations for the atmosphere and the ocean and for primitive equations exponential decay to equilibria is known, see e.g. \cite{HK:16,Iid:25,Bin:25}. 
It is therefore not surprising that the mechanism leading to the polynomial decay stems from the non-linear coupling conditions \eqref{eq:wind bc}. 
To understand this effect better let us consider the following 1D toy model
\begin{equation}
	\left\{
		\begin{aligned}
			\partial_t f(t,p) = f''(p), \quad &\text{ on } (0,1), \\
			\partial_t g(t,z) = g''(z), \quad &\text{ on } (-1,0), \\
		\end{aligned}
	\right.
	\label{eq:toy}
\end{equation}
supplemented to the boundary conditions
\begin{equation*}
	\begin{aligned} 
	f'(0) = 0 \quad &\text{ and } \quad 
	f'(1) = -|f(1) -g(0)|\cdot(f(1) -g(0)) , \\ 
	g'(-1) = 0 \quad &\text{ and } \quad g'(0) = 
	 |f(1) -g(0)|\cdot(f(1) -g(0)) ,
	\end{aligned} 
\end{equation*}
and the initial data $(f,g)|_{t = 0} = (f_0,g_0)$. 
It shares some features with the \eqref{eq:CAO}-model: it is easy to see that the energy is a strict Lyapunov functional and that the equilibria are constants $(c,c)$. On the other hand, considering the \emph{dissipation} given by
\begin{equation*}
	\mathrm{D}(t)	
	:= \| f' \|_{\rL^2(0,1)}^2
	+ \| g' \|_{\rL^2(-1,0)}^2 
	+ |f(1) -g(0) |^3
\end{equation*}
we see that it is a gradient flow 
\begin{equation*}
	\partial_t \binom{f}{g}
	= - \binom{\frac{\delta D}{\delta f}}{\frac{\delta D}{\delta f}}
\end{equation*}
of the dissipation $\mathrm{D}$ on $\rL^2(0,1) \times \rL^2(-1,0)$. Now, the (generalized) Lojasiewicz-Simon inequality yields
\begin{equation*}
	|\mathrm{D}-\mathrm{D}_\infty|^{1-\frac{1}{3}} \leq 
	C \cdot \| f'' \|_{\rL^2} + \| g'' \|_{\rL^2} = \| \nabla \mathrm{D} \|_{\rL^2}
\end{equation*}
near an equilibrium $(c,c)$, due to the mismatch between the exponents $2$ and $3$. 
The gradient structure implies 
\begin{equation}
	\partial_t \mathrm{D}(t) \leq 
	- C \cdot \mathrm{D}(t)^{2(1-\frac{1}{3})} = - C \cdot \mathrm{D}(t)^{\frac{4}{3}} 
	\label{eq:D intro}
\end{equation}
and hence principle of ODEs implies  
\begin{equation*}
	\mathrm{D}(t) \leq C \cdot (1+t)^{-3} .
\end{equation*}
In order to obtain the $\rL^2$- or energy decay rate we note that the result above yields
\begin{equation*}
	\| f' \|_{\rL^2} + 
	\| g' \|_{\rL^2} \leq C \cdot (1+t)^{-{\frac{3}{2}}} \quad \text{ and } \quad 
	|f(1)-g(0)| \leq C \cdot (1+t)^{-1} .
\end{equation*}
For fixed total momentum we obtain a Poincaré inequality\footnote{Indeed, such a Poincaré type inequality holds for the CAO-system, see \autoref{prop:poincare inequalitly}.} which penalizes the mismatch of the boundary conditions in the right way
\begin{equation*}
	\| f \|_{\rL^2}^2 + 
	\| g \|_{\rL^2}^2 
	\leq C \cdot ( 	\| f' \|_{\rL^2}^2 + 
	\| g' \|_{\rL^2}^2 + |f(1)-g(0)|^2 )
	\leq C \cdot (1+t)^{-3} + C \cdot (1+t)^{-2} \leq C' \cdot (1+t)^{-2} .
\end{equation*}
Hence, the energy or $\rL^2$-convergence rate matches exactly the convergence rate of the CAO-system from \autoref{thm:long-term}.
This is not a coincidence! 
In fact for initial data $(\vair_0,\vocn_0)$ which are $x,y$ independent, the CAO-system yield a slightly modified version of our toy example, see \autoref{exa:algebraic decay}. This confirms the optimality of the polynomial decay rates in \autoref{thm:long-term}.

\smallskip 

It is easy to see from our toy problem that the polynomial decay is a purely non-linear effect. Indeed, replacing the coupling boundary conditions by
\begin{equation*}
	f'(1) = -|f(1)-g(0)|
	, \qquad \text{ and } \qquad
	g'(0) = |f(1)-g(0)| ,
\end{equation*}
would lead to the standard exponent $\theta = \nicefrac{1}{2}$ in the Lojasiewicz-Simon inequality and therefore to exponential decay.

\medskip 

Let us now explain the main ideas of the proof of the main \autoref{thm:long-term}.
By \autoref{thm:equilibria} we are able to identify which equilibrium is approached by which trajectory, if the trajectory converges. 
With this preparation we come now to the core of the proof: establishing the optimal polynomial convergence rate in $\rH^1$. 

The first step to do so is to construct a neighbourhood of the prescribed equilibrium where the global solution is trapped, see \autoref{prop:trapping}. This requires several ingredients:
the first one is a Poincaré inequality which penalizes differences in the coupling by a carefully chosen rate, see \autoref{prop:poincare inequalitly},
the second one is the Hessian-Laplacian identity, \autoref{lem:bochner formula}, which contains a positive non-linear, boundary term. Using the subcriticality of this boundary term we are able to show equivalence of the $\rH^2$-norm and the graph norm of the Laplacians near the equilibrium, see \autoref{cor:elliptic estimate}.
Further, we need a non-linear elliptic estimate near the equilibrium, \autoref{lem:elliptic estimate 2}.
Finally, motivated by our toy example we consider the \textit{dissipation} of the CAO-system by
\begin{equation*}
	\mathrm{D}(t) := 
	\| \nabla^a \vair \|_{\rL^2(\Omegaair)}^2 
	+ \| \nabla^a \vocn \|_{\rL^2(\Omegaocn)}^2 +
	p_s \|\vatm|_{\Gammaairu} -\vocn|_{\Gab}\|_{\rL^3(\T^2)}^3 .
\end{equation*}
Using anisotropic estimates, interpolation theory and Sobolev embeddings, we establish a differential inequality for it in \autoref{lem:differential inequality dissipation}.
Now, \autoref{prop:trapping} follows by a continuation argument. 

\smallskip 

While we cannot use the Lojasiewicz-Simon inequality here, we nevertheless aim for the a similar differential inequality as \eqref{eq:D intro}.
Also in contrast to the toy model where the $\rL^2$-rate was concluded from the $\rL^2$-rate of the gradients by a Poincaré inequality, we do it here vice-versa: we first establish the $\rL^2$-convergence rate and use it to conclude the convergence rate of the gradients. 

Inside the trapping region around the prescribed equilibrium our Poincaré inequality \autoref{prop:poincare inequalitly} yields the energy inequality, see \autoref{lem:differential inequality},
\begin{equation}
	\partial_t \rE(t) + C \cdot \rE(t)^{\frac{3}{2}} \leq 0
	\label{eq:energy} 
\end{equation}
and finally the convergence in the $\rL^2$-norm with the desired convergence rate, see \autoref{lem:algebraic convergence rate L^2}.
Observe that the differential inequality \eqref{eq:energy} matches exactly the one predicted from the toy model. 
The deeper reason for this is that on the energy level the non-linearities and non-local terms do not enter due to the cancellation law and the incompressibility.

This is no longer true on the level of gradients and makes the the derivation of the convergence of gradients and its convergence rate more involved. 

The key estimate is the differential inequality
\begin{equation*}
	\partial_t \mathrm{D}(t)
	+ c_0 \cdot \mathrm{D}(t)^{\frac{4}{3}} \leq C \cdot (1+t)^{-2} . 
\end{equation*}
near the prescribed equilibrium, which mirrors \eqref{eq:D intro} but with an error term. 
The error term on the right-hand side comes from the non-linear terms and from the non-local terms. 
By a barrier function argument we obtain an upper bound for the decay rate $\mathcal{O}((1+t)^{-\frac{3}{4}})$ of the gradients, see \autoref{prop:algebraic convergence rate small}.
Unfortunately, as we can see in our toy problem this rate is not optimal.  

A closer investigation of the error term shows that it has, roughly speaking, the following structure
\begin{equation*}
	\mathrm{error} \sim  
	\| v \|_{\rL^2}^4 
	+ \| v \|_{\rL^2}^2 \cdot 
	\| \nabla v \|_{\rL^2}^2 .
\end{equation*}
Note that the first term decay like $\mathcal{O}((1+t)^{-4})$ by \autoref{lem:algebraic convergence rate L^2}, which by the structure of the differential inequality is exactly the saturation point: faster decay of the right-hand side than that does not yield to faster decay of the dissipation $\mathrm{D}$. 
While in the proof of \autoref{prop:algebraic convergence rate small}, we merely used for 
$\| \nabla v \|_{\rL^2}^2$ that it is bounded (and small), we have now more detailed information about it thanks to \autoref{prop:algebraic convergence rate small} itself. Therefore, we obtain a better decay rate of the error term.
Now, a boot-strap argument leads to the optimal decay rate, see \autoref{prop:final algebraic convergence rate small}. The argument saturates in finitely many steps.

\subsection{Structure of the article}

\

In \autoref{sec:prelim} we recall necessary preliminary result. 
In \autoref{sec:dynamics}, we investigate the dynamical properties of the system. 
This section is divided into two parts: 
In \autoref{sec:equilibria}, we perform the calculations which yield to the first half of \autoref{thm:equilibria}: the equilibria and the associated total energy, as well as that fact that the equilibria are not normally stable and the dynamics process a four-dimensional center manifold with non-trivial dynamics. Additional, we show that trajectories cannot stop, \autoref{cor:not finite time} and that the semi-flow has no nontrivial recurrent points. Both are easy consequences from \autoref{prop:energy lyapunov} and the analyticity of the trajectory. Finally, in \autoref{sec:phase space}, we show that the total momentum is conserved, and use this property to describe the structure of the phase space.

\autoref{sec:asymptotics} is devoted to the proof of the main result \autoref{thm:long-term}. 
In \autoref{ssec:apriori} we recall the a-priori estimates from \cite{BBHZ:25} and prove our key Poincare inequality \autoref{prop:poincare inequalitly}. In \autoref{ssec:trapping} we construct the trapping region, see \autoref{prop:trapping}. 
Finally, in \autoref{ssec:convergence rate} we establish convergence in $\rH^1$ to the prescribed equilibrium and determine the optimal rate. 

The section concludes with the long-term behaviour of the CAO-system for two particular initial data: 
In \autoref{exa:algebraic decay} we show that for certain initial data the CAO-system collapses to (a modified version) of our Toy example, which justifies the optimality of the decay rate.
In \autoref{exa:exponential decay} we demonstrate that, for certain initial data, exponential decay can be achieved.

\section{Preliminaries}
\label{sec:prelim}

In this section we recall some preliminary results we need in the sequel. 

\smallskip 

For ${v \colon \T^2 \times (a,b) \rightarrow \R^2}$, we denote by~$\overline{v} := \frac{1}{b-a} \int_{a}^{b} \divH v(\cdot , \cdot , \xi) \d \xi$ the vertical average.
Following \cite{CT:07,HK:16} we define 
\emph{hydrostatically solenoidal vector fields} by
\begin{equation*}
	\rL^2_{\sigmabar}(\Omega) = \overline{\{v \in \rC^\infty(\overline{\Omega};\R^2) \colon \divH \vbar = 0 \}}^{\| \cdot \|_{\rL^q(\Omega)}}
\end{equation*}
on $\Omega \in \{ \Omegaatm, \Omegaocn \}$. Its role for the primitive equations is analogous to the one of the solenoidal vector fields for the Navier-Stokes equations. 
We denote $\rH^s(\Omega)$ for $s \in \R$ the Bessel potential spaces, and define $\rH^s_{\sigmabar}(\Omega) := \rH^s(\Omega) \cap \rL^2_{\sigmabar}(\Omega)$ their hydrostatically solenoidal counterparts. 
Furthermore, we denote the jump at the interface by 
\begin{equation*} 
	[v] := \vair|_{\Gammaairu} - \vocn|_{\Gab}.
\end{equation*}  

Using the incompressibility and the boundary conditions, the vertical velocity $w$ is determined uniquely by the horizontal velocity $v$ by
\begin{equation}
	w(x,y,z)
	= - \int_{-h}^z \divH v(x,y,\xi) \, \d \xi . 
		\label{eq:wv}
\end{equation}
The geopotential and the pressure are determined by
\begin{align}
	\nablaH \Phi 
	&= (\nablaH (-\DeltaH)^{-1} \divH) ( \overline{\vair \cdot \nablaH \vair + w^{\air} \partial_p \vair} ) + p_s (\nablaH (-\DeltaH)^{-1} \divH) ( [v] | [v] | ) \\
	\nablaH \pi 
	&= (\nablaH (-\DeltaH)^{-1} \divH) ( \overline{\vocn \cdot \nablaH \vocn + w^{\ocn} \partial_z \vocn } ) + p_s (\nablaH (-\DeltaH)^{-1} \divH) ( [v] | [v] | ) .
	\label{eq:pressure}
\end{align}
Note that they are only determined up to constants. In order to fix them by assuming in the sequel that they are mean-free $\int_{\T^2} \Phi_s = 0$ and $\int_{\T^2} \pi_s = 0$.

\section{Dynamics}
\label{sec:dynamics}

This section is dedicated to the investigation of the dynamical aspects of the \eqref{eq:CAO}-system. 

\subsection{Global semi-flow}
\label{ssec:semi-flow}

\

In this section we summarize the results from \cite{BBHZ:25} for the situation of under consideration.
The main result from \cite[Theorem 3.1]{BBHZ:25} for $p = q = 2$ reads as follows. 

	\begin{thm}[Almost global strong well-posedness for the CAO-system]\label{thm:globalwellposed} 
	Let $T > 0$ and $(\vair_0, \vocn_0)^\top \in \rH^1_{\sigmabar}(\Omegaatm) \times 
	\rH^1_{\sigmabar}(\Omegaocn)$. 
	Then there is a unique, strong solution $(\vair,\omegaair,\Phi)$, $(\vocn, \wocn, \pi)$ to \eqref{eq:CAO}, subject to \eqref{eq:wind bc}, satisfying
	\begin{equation*}
		\begin{aligned}
			\vair &\in 
			\rH^{1}(0,T;\rL^2_{\sigmabar}(\Omegaair;\R^2)) \cap \rL^2(0,T;\rH^{2}(\Omegaair;\R^2)), \\
			\vocn &\in 
			\rH^{1}(0,T;\rL^2_{\sigmabar}(\Omegaocn;\R^2)) \cap \rL^2(0,T;\rH^{2}(\Omegaocn;\R^2)).
		\end{aligned}
	\end{equation*}
\end{thm}

Using the embeddings
\begin{equation*}
	\begin{aligned}
		\rH^{1}(0,T;\rL^2_{\sigmabar}(\Omegaair;\R^2)) \cap \rL^2(0,T;\rH^{2}(\Omegaair;\R^2)) 
		&\hookrightarrow
		\rC([0,T];\rH^1_{\sigmabar}(\Omegaatm;\R^2)) \\
		\rH^{1}(0,T;\rL^2_{\sigmabar}(\Omegaocn;\R^2)) \cap \rL^2(0,T;\rH^{2}(\Omegaocn;\R^2))
		&\hookrightarrow
		\rC([0,T];\rH^1_{\sigmabar}(\Omegaocn;\R^2))
	\end{aligned}
\end{equation*}
and the fact that a continuity is a local property, we conclude the following dynamical interpretation of our previous result. 

\begin{cor}[Global semi-flow on phase space]
	For every $(\vair_0, \vocn_0)^\top \in \rH^1_{\sigmabar}(\Omegaatm) \times 
	\rH^1_{\sigmabar}(\Omegaocn)$ there exists a unique, strong solution
	\begin{equation*}
		(\vair,\vocn)^\top \in \rC([0,+\infty);\rH^1_{\sigmabar}(\Omegaatm;\R^2)\times\rH^1_{\sigmabar}(\Omegaocn;\R^2)) .
	\end{equation*} 
	Furthermore, the flow map 
	\begin{equation*}
		S(t) \colon \rH^1_{\sigmabar}(\Omegaatm;\R^2)\times\rH^1_{\sigmabar}(\Omegaocn;\R^2)
		\to \rH^1_{\sigmabar}(\Omegaatm;\R^2)\times\rH^1_{\sigmabar}(\Omegaocn;\R^2) \quad \text{ with } 
		\quad S(t) (\vair_0, \vocn_0)^\top 
		:= (\vair(t),\vocn(t))^\top 
	\end{equation*}
	is a continuous global semi-flow. 
\end{cor}

In fact, using Angenent's parameter trick \cite{Ang:90a,Ang:90b} and repeating the arguments from \cite[Section 7]{BBHZ:25} (only for time instead of time and space)
we can upgrade the continuity in time to real analyticity.

\begin{cor}[Real analyticity in time]\label{cor:analyticity}
	The global semi-flow is real analytic in time. Hence for every $(\vair_0, \vocn_0)^\top \in \rH^1_{\sigmabar}(\Omegaatm) \times 
\rH^1_{\sigmabar}(\Omegaocn)$ the unique strong solution satisfies
\begin{equation*}
	(\vair,\vocn)^\top \in \rC^{\omega}((0,+\infty);\rH^1_{\sigmabar}(\Omegaatm;\R^2)\times\rH^1_{\sigmabar}(\Omegaocn;\R^2)) .
\end{equation*} 	
\end{cor}

\subsection{Equilibria}
\label{sec:equilibria}

\

In this section we study equilibria of the semi-flow generated by \eqref{eq:CAO}.
A equilibrium of the semi-flow generated by \eqref{eq:CAO} is a steady state of \eqref{eq:CAO}, i.e. a solution to 
	\begin{equation}
	\left\{
	\begin{aligned}
		\vair \cdot \nablaH \vair + w^{\air} \partial_p \vair  + \nablaH \Phi_s &= \Delta^{\atm} \vair, &&\text{ on } \Omegaair \times (0,T),  \\
		\divH \vairbar &= 0, &&\text{ on } \Omegaair \times (0,T),\\
		\vocn \cdot \nablaH \vocn + w^{\ocn} \dz \vocn + \nablaH \pi_s &= \Delta \vocn, &&\text{ on } \Omegaocn \times (0,T), \\
		\divH \vocnbar &= 0, &&\text{ on } \Omegaocn \times (0,T), \\
	\end{aligned}
	\right. 
	\label{eq:equilibrium}
\end{equation}
supplemented by the boundary conditions 
\begin{equation*} 
	\begin{aligned} 
		(\partial_p \vatm)|_{\Gammaairb} = 0 \quad &\text{ and } \quad (\partial_p \vatm)|_{\Gammaairu} = -p_s^{-1} \cdot |\vatm|_{\Gammaairu} -\vocn|_{\Gab}|\cdot(\vatm|_{\Gammaairu} -\vocn|_{\Gab} ) , \\ 
		(\partial_z \vocn)|_{\Gab} = 0 \quad &\text{ and } \quad (\partial_z \vocn)|_{\Gau} = p_s \cdot  |\vatm|_{\Gammaairu} -\vocn|_{\Gab}|\cdot  (\vatm|_{\Gammaairu} -\vocn|_{\Gab}) .
	\end{aligned} 
\end{equation*}

The equilibria of the semi-flow generated by \eqref{eq:CAO} can be characterized as follows. 

\begin{prop}\label{prop:steady states}
	The set $\sE$ of equilibria of \eqref{eq:CAO} forms a two-dimensional vector-space 
	\begin{equation*}
		\sE
		= \{ (c_*,c_*) \colon c_* \in \R^2 \} .
	\end{equation*} 
	In particular, $\sE$ is a real analytic two-dimensional manifold. 
\end{prop}
\begin{proof}
	It is easy to see that $(c_*,c_*)$ for $c_* \in \R^2$ is an equilibrium. For the other direction we obtain by testing \eqref{eq:equilibrium}$_1$ by $\vair$ and \eqref{eq:equilibrium}$_3$ by $\vocn$, integrating by parts and adding both equations that
	\begin{align*}
		0 &= \int_{\Omegaair} (|\nablaH \vair|^2 + |p \partial_p \vair|^2 )
		+ \int_{\Omegaocn} |\nabla \vocn|^2 
		-\int_{\T^2} p_s^2 \vair \partial_p \vair -\int_{\T^2} \vocn \partial_z \vocn \\
		&= \int_{\Omegaair} (|\nablaH \vair|^2 + |p \partial_p \vair|^2) 
		+ \int_{\Omegaocn} |\nabla \vocn|^2 
		+\int_{\T^2} p_s |[v]|^3 .
	\end{align*}
	This implies $\vair \equiv c_*, \vocn \equiv c_\dagger$ with $c_*,c_\dagger \in \R^2$, as well as $c_* = \vair_{\Gammaairu} = \vocn_{\Gab} = c_\dagger$, i.e. $(\vair,\vocn) = (c_*,c_*)$ for $c_* \in \R^2$.   
\end{proof}

Now, we consider the \emph{total energy} of \eqref{eq:CAO} given by
\begin{equation*}
	\rE(t)
	:= \frac{1}{2}\int_{\Omegaatm} |\vair(t)|^2 + \frac{1}{2}\int_{\Omegaocn} |\vocn(t)|^2 
\end{equation*}
and its dissipation
\begin{equation*}
	\mathrm{D}(t) := 
	\| \nabla^a \vair \|_{\rL^2(\Omegaair)}^2 
	+ \| \nabla^a \vocn \|_{\rL^2(\Omegaocn)}^2 +
	p_s \|[v]\|_{\rL^3(\T^2)}^3  .
\end{equation*}
Note that the total energy consists of the kinetic energies of the atmosphere and the ocean. It is a strict Lyapunov functional of the semi-flow generated by \eqref{eq:CAO}.

\begin{prop}\label{prop:energy lyapunov}
	We have
	\begin{equation*}
		\partial_t \rE(t) = - \mathrm{D}(t) 
	\end{equation*}
	and hence the total energy $\rE$ is a strict Lyapunov functional.
\end{prop}
\begin{proof}
	Using the wind-driven boundary conditions, integration by parts yields
	\begin{align*}
		\partial_t \rE(t)
		&= \int_{\Omegaatm} \vair \cdot \partial_t \vair  + \int_{\Omegaocn} \vocn \cdot \partial_t \vocn 
		\\
		&= \int_{\Omegaatm} \vair \cdot (\DeltaH \vair + \partial_p(p^2\partial_p \vair) - \vair \cdot \nablaH \vair - \omega \partial_p \vair - \nablaH \Phi) \\ 
		&+ \int_{\Omegaocn} \vocn \cdot (\Delta \vocn - \vocn \cdot \nablaH \vocn - w \partial_z \vocn - \nablaH \pi) 
		\\
		&= - \int_{\Omegaatm} |(\nablaH \vair|^2 + |p \partial_p \vair|^2) - \int_{\Omegaocn} |\nabla \vocn|^2 +\int_{\T^2} p_s^2 \vair \partial_p \vair +\int_{\T^2} \vocn \partial_z \vocn  
		\\
		&= - \int_{\Omegaatm} (|\nablaH \vair|^2 
		+ |p \partial_p \vair|^2 ) - \int_{\Omegaocn} |\nabla \vocn|^2 \\
		&-\int_{\T^2} p_s \vair \cdot ([v]) |[v]| + p_s \int_{\T^2} \vocn \cdot ([v]) |[v]| \\
		&= - \int_{\Omegaatm} (|\nablaH \vair|^2 
		+ |p \partial_p \vair|^2) - \int_{\Omegaocn} |\nabla \vocn|^2 -\int_{\T^2} p_s ([v]) \cdot ([v]) |[v]| \\
		&= - \int_{\Omegaatm} (|\nablaH \vair|^2 
		+ |p \partial_p \vair|^2) - \int_{\Omegaocn} |\nabla \vocn|^2 -\int_{\T^2} p_s |[v]|^3 = - \mathrm{D}(t)
		\leq 0 .
	\end{align*}
	Furthermore, we see from this calculation that 
	\begin{equation*}
		0 = \partial_t E(t)
		= - \int_{\Omegaatm} (|\nablaH \vair|^2 
		+ |p \partial_p \vair|^2) - \int_{\Omegaocn} |\nabla \vocn|^2 -\int_{\T^2} p_s |[v]|^3 ,
	\end{equation*}
	this is equivalent to $(\vair,\vocn) = (c_*,c_*)$ as we have seen in the proof of \autoref{prop:steady states}.
\end{proof}

The total linearization of \eqref{eq:CAO} at an equilibrium $(c_*,c_*)$ is given by $\bar{\rA}_c \colon D(\bar{\rA}_c) \subset \rL^2_{\sigmabar}(\Omegaair) \times \rL^2_{\sigmabar}(\Omegaocn) \to \rL^2_{\sigmabar}(\Omegaair) \times \rL^2_{\sigmabar}(\Omegaocn)$ with
\begin{equation}
	\begin{aligned}
		\bar{\rA}_{c_*} 
		&:=
		\begin{pmatrix}
			\DeltaH + \partial_p p^2 \partial_p - c_* \cdot \nablaH & 0 \\
			0 & \Delta - c_* \cdot \nablaH
		\end{pmatrix}
		\\
		D(\bar{\rA}_c)
		&:= \left\{ \binom{\hvair}{\hvocn} \in \rH^2_{\bar{\sigma}}(\Omegaair)\times \rH^2_{\bar{\sigma}}(\Omegaocn) \colon \partial_n \hvair|_{\Gammaairb} = \partial_n \hvair|_{\Gammaairu} = 0 \ \text{ and } \ \partial_n \hvocn|_{\Gau} = \partial_n \hvocn|_{\Gab} = 0 \right\} .
	\end{aligned} 
\end{equation}
Here, we have used the fact that the hydrostatic Stokes operator with Neumann boundary conditions coincide with the part of the Laplacian with Neumann boundary conditions in $\rL^2_{\sigmabar}$, see \cite{GGHHK:17}. 
We denote by $\rX_1 := D(\bar{\rA}_{c_*})$ equipped with the graph norm and $\rX_0 := \rL^2_{\sigmabar}(\Omegaair) \times \rL^2_{\sigmabar}(\Omegaocn)$. 
The following result is a direct consequence of the main result of \cite{GGHHK:17}.

\begin{lem}\label{lem:Abar mr}
	The operator $\bar{\rA}_{c_*}$ admits bounded $\rH^\infty$-calculus on $\rX_0$. In particular, it has maximal regularity on $\rX_0$ and generate an analytic semigroup $(T_{\bar{\rA}_{c_*}}(t))_{t \geq 0}$ on $\rX_0$.  
\end{lem}

Moreover, we obtain the following spectral results for the linearization.

\begin{lem}\label{lem:Abar spectrum}
	The operator $\bar{\rA}_{c_*}$ has compact resolvent on $\rX_0$. Furthermore, we obtain 
	\begin{equation*}
		\sigma (\bar{\rA}_{c_*}) = \sigma_p(\bar{\rA}_{c_*})
		= \{0\} \, \dot{\cup} \,  
		\biggl\{ - \frac{1}{4} - \frac{(\pi m)^2}{(\log(p_a/p_s))^2} -\frac{(\pi n)^2 }{h^2} + |k|^2 + i c \cdot k \colon k \in \Z^2, \ n, m \in \N_0 \biggr\} 	
	\end{equation*}
	and the eigenspaces of $\lambda \not = 0$ are generated by
	\begin{equation}
		\binom{-k_2}{k_1} e^{i k \cdot (x,y)} 
		p^{-\frac{1}{2}} 
		\left(
			\sin\biggl(\frac{\pi m \log(p/p_s)}{\log(p_a/p_s)}\biggr)
			+ \frac{2 n \pi}{\log(p_a/p_s)} \cdot \cos\biggl(\frac{\pi m \log(p/p_s)}{\log(p_a/p_s)}\biggr)
		\right) 
		\label{eq:eigenfunctions}
	\end{equation}
	and
	\begin{equation}
		\binom{-k_2}{k_1} e^{i k \cdot (x,y)} \cos\biggl(\frac{\pi n z}{h}\biggr) .
		\label{eq:eigenfunctions2}
	\end{equation}
	Finally, the non-trivial kernel of $\bar{\rA}_{c_*}$ is given by 
	\begin{equation*}
		\ker(\bar{\rA}_{c_*})
		= \{ (c_1,c_2)^\top \colon c_1,c_2 \in \R^2 \}
	\end{equation*}  
	and has dimension four. 
\end{lem}
\begin{proof}
	By Sobolev embedding we obtain $D(\bar{A}_c) \hookrightarrow \rL^2_{\sigmabar}(\Omegaair) \times \rL^2_{\sigmabar}(\Omegaocn)$ and therefore $\bar{A}_c$ has compact resolvent. This implies that $\sigma(\bar{A}_c) = \sigma_p(\bar{A}_c)$. The remaining part follows by direct calculations.
\end{proof}

\begin{rem}
	Note that $\dim T_{(c_*,c_*)} = \dim \sE = 2 \not = 4 = \dim \ker\bar{A}_{c_*}$ and the equilibrium $(c_*,c_*)$ is not normally stable.
	Therefore, the generalized principle of linearized stability \cite{PSZ:09,PS:16} does not apply.
	
	More precisely, $\dim \ker\bar{A}_{c_*} = 4$ shows that the dynamics has a $4$-dimensional center manifold which is a slow manifold, since $\i \R \cap \sigma(\bar{A}_{c_*}) = \{ 0 \}$. Since $\dim \sE = 2$ there is a non-trivial dynamics on the center manifold. 
	This is in sharp contrast to the situation of primitive equations.
\end{rem}

We finish this section by studying the behaviour of the orbits.

\begin{cor}[Equilibria cannot be reached in finite time]\label{cor:not finite time}
	Assume $(\vair_0,\vocn_0) \not \in \sE$, then the solution satisfies $(\vair(t),\vocn(t)) \not \in \sE$ for all $t \in [0,+\infty)$.
\end{cor}
\begin{proof}
	Assume that there exists a point in time $t_0 \in [0,+\infty)$, where $(\vair(t_0),\vocn(t_0)) \not \in \sE$. Then there exists a constant $c_* \in \R^2$ such that $(\vair(t),\vocn(t)) = (c_*,c_*)$ for all $t \in [t_0,+\infty)$. But by \autoref{cor:analyticity} the solution
	\begin{equation*}
		[0,+\infty) \to \rH^1(\Omegaair) \times \rH^1(\Omegaair) \colon t \mapsto (\vair(t),\vocn(t))
	\end{equation*}
	is real-analytic and therefore the identity theorem implies that $(\vair(t),\vocn(t)) = (c_*,c_*)$ for all $t \in (0,+\infty)$. 
\end{proof}

This result implies that the semi-flow $S(t)$ has no non-trivial recurrent points.

\begin{cor}[No recurrent points]\label{cor:no recurrence}
	Consider $(\vair_0,\vocn_0)^\top \in H_\gamma$ with $(\vair_0,\vocn_0)^\top \not \in \mathcal{E}$. Then the map
	\begin{equation*}
		[0,+\infty) 
		\to \rX_\gamma 
		\colon t \mapsto 
		(\vair(t),\vocn(t))^\top 
	\end{equation*}
	is injective. 
\end{cor} 
\begin{proof}
	We assume that there exists $t_1 < t_2$ such that $(\vair(t_1),\vocn(t_1)) = (\vair(t_2),\vocn(t_2))$.
	Uniqueness implies 
	\begin{equation*}
		(\vair(t_1+t),\vocn(t_1+ t)) = (\vair(t_2+t),\vocn(t_2+t))
		\quad \text{ for } t \geq 0  
	\end{equation*}
	and hence the solution is time-periodic. Since the energy $\rE$ is a Lyapunov functional, it implies that $\rE(t)$ is constant for $t \geq t_1$. Since the Lyapunov functional $\rE$ is strict, we see that $(\vair(t),\vocn(t))^\top \in \mathcal{E}$ for $t \geq t_1$. Hence, our solution has reached an equilibrium in finite time, which by \autoref{cor:not finite time} is only possible if $(\vair_0,\vocn_0)^\top \in \mathcal{E}$.
\end{proof}

Note that this in particular shows that the orbits are either points (if one has an equilibrium) or infinite, analytic, embedded, non-self-intersecting curves in the phase space $\rX_\gamma$.

\subsection{Structure of the Phase space}
\label{sec:phase space}

\

In this section we study the structure of the phase space of \eqref{eq:CAO}. This will allow us to identify the limit points depending on the initial data. We recall that the phase space is given by
\begin{equation*}
	\rX_{\gamma}
	= \rH^1_{\sigmabar}(\Omegaatm;\R^2)\times\rH^1_{\sigmabar}(\Omegaocn;\R^2).
\end{equation*} 

\smallskip 

We consider the \emph{total momentum} of the system is given by\footnote{Note that we have assumed for simplicity $\rhoair=1$ and $\rhoocn=1$.}
\begin{equation*}
	\mathrm{P}(t) := \mathrm{P}((\vair(t),\vocn(t)))
	:= \int_{\Omegaair} \vair(t) 
	+ \int_{\Omegaocn} \vocn(t) .
\end{equation*}
Furthermore, we introduce the \emph{total mass} of the system as
\begin{equation*}
	\mathrm{M} := \int_{\Omegaair} 1 + \int_{\Omegaocn} 1 = |\Omegaair| + |\Omegaocn|.
\end{equation*}  
Trivially, the total mass is conserved. 
Now, the \emph{averaged velocity} is given by
\begin{equation*}
	\mathrm{V}(t) := \mathrm{V}((\vair(t),\vocn(t)))
	:= \frac{\mathrm{P}(t)}{\mathrm{M}} .
\end{equation*}

\begin{prop}\label{prop:PV conserved}
	The total momentum and the averaged velocity are conserved, i.e. $\partial_t \mathrm{P}(t) = \partial_t \mathrm{V}(t) = 0$.
\end{prop}
\begin{proof}
	It suffices to show that $\partial_t \mathrm{P}(t) = 0$. 
	Using the boundary conditions, we obtain by the divergence theorem
	\begin{align*}
		\partial_t \mathrm{P}(t)
		&=  \int_{\Omegaair} \partial_t \vair 
		+ \int_{\Omegaocn} \partial_t \vocn \\
		&= \int_{\Omegaair} \div (\tbinom{\nablaH \vair}{p^2 \partial_p \vair} - \uair \otimes \vair + \Phi I_3)
		+ \int_{\Omegaocn} \div (\nabla \vocn - \uocn \otimes \vocn + \pi I_3) \\
		&= \int_{\Gammaairu} (p_s^2 \cdot \partial_p \vair + \Phi e_3) 
		+ \int_{\Gau} (\partial_z \vocn + \pi e_3)  \\
		&= \int_{\Gammaairu} p_s^2 \cdot \partial_p \vair 
		+ \int_{\Gau} \partial_z \vocn 
		+ 
		\int_{\T^2} \Phi e_3
		+ 
		\int_{\T^2} \pi e_3 = 0.  
		\qedhere 
	\end{align*}
\end{proof}

\autoref{prop:PV conserved} implies that the phase space $\rX_\gamma = \rH^1_{\sigmabar}(\Omegaatm;\R^2)\times\rH^1_{\sigmabar}(\Omegaocn;\R^2)$ is a trivial fibre bundle 
\begin{equation*}
	\rX_{\gamma}
	=
	\bigsqcup_{c_* \in \R^2} \mathrm{V}^{-1}(c_*) \cong \R^2 \times \mathrm{V}^{-1}(c_*).
\end{equation*} 
Further, each fibre $\rX_{c_*} := \mathrm{V}^{-1}(c_*)$ is a affine-linear space of codimension two in $\rX_\gamma$ and contains exactly one equilibrium $(c_*,c_*)$.

\section{Asymptotics}
\label{sec:asymptotics}

In this section we study the dynamics along each fibre $\rX_{c_*} = \mathrm{V}^{-1}(c_*)$ for $c_* \in \R^2$, i.e. the restriction of the semi-flow $S(t)|_{\rX_{c_*}}$. 
Since $\rX_{c_*}$ are affine-linear spaces, we obtain the following Banach space by subtracting the reference point $(c_*,c_*)^\top$
\begin{equation}
	\rH_\gamma
	:= 
	\rX_{c_*} - (c_*,c_*)^\top 
	= 
	\left\{ \binom{\hvair}{\hvocn} \in  \rH^1_{\sigmabar}(\Omegaair) \times \rH^1_{\sigmabar}(\Omegaocn)
	\colon 
	V(\hvair,\hvocn)
	= 0
	\right\} .
	\label{eq:phase space fibre}
\end{equation} 
Hence, the evolution along a fibre is given by 
	\begin{equation}
	\left\{
	\begin{aligned} 
		\dt \hvair+ \hvair \cdot \nablaH \hvair +  c_* \cdot \nablaH \hvair + \hat{\omega} \partial_p \vair  + \nablaH \hat{\Phi}_s &= \Delta^{\air} \hvair, &&\text{ on } \Omegaair \times (0,T),  \\
		\divH \bar{\hvair} &= 0, &&\text{ on } \Omegaair \times (0,T),\\
		\dt \hvocn  + \hvocn \cdot \nablaH \hvocn + c_* \cdot \nablaH \hvocn + \hat{w} \dz \hvocn + \nablaH \hat{\pi}_s &= \Delta \hvocn, &&\text{ on } \Omegaocn \times (0,T), \\
		\divH \bar{\hvocn} &= 0, &&\text{ on } \Omegaocn \times (0,T),
	\end{aligned}
	\right. 
	\label{eq:CAOc}
	\tag{$\text{CAO}_{\text{c}}$}
\end{equation}
with $\Delta^{\air} = \DeltaH + \partial_p (p^2 \partial_p )$, and complemented by the boundary conditions
\begin{equation*} 
	\begin{aligned} 
		(\partial_p \hvair)|_{\Gammaairb} = 0 \quad &\text{ and } \quad (\partial_p \hvair)|_{\Gammaairu} = -p_s^{-1} \cdot |\hvair|_{\Gammaairu} -\hvocn|_{\Gab}|\cdot(\hvair|_{\Gammaairu} -\hvocn|_{\Gab} ) , \\ 
		(\partial_z \hvocn)|_{\Gab} = 0 \quad &\text{ and } \quad (\partial_z \vocn)|_{\Gau} = p_s \cdot  |\hvair|_{\Gammaairu} -\hvocn|_{\Gab}|\cdot  (\hvair|_{\Gammaairu} -\hvocn|_{\Gab}) .
	\end{aligned} 
\end{equation*}
as well as the initial data
\begin{equation*}
	(\hvair, \hvocn)|_{t=0} = (\hvair_0, \hvocn_0) 
\end{equation*}
on the Banach space
\begin{equation*}
	\rH_0 
	= 
	\left\{ \binom{\hvair}{\hvocn} \in  \rL^2_{\sigmabar}(\Omegaair) \times \rL^2_{\sigmabar}(\Omegaocn)
	\colon 
	\frac{\int_{\Omegaair} \hvair
	+ \int_{\Omegaocn} \hvocn}{|\Omegaair|+|\Omegaocn|} 
	= 0
	\right\} .
\end{equation*}
Note that the equilibrium $(c_*,c_*)^\top \in \rX_{c_*}$ corresponds to the equilibrium $(0,0) \in \rH_\gamma$. 
It is unique. 
The (total) linearization of \eqref{eq:CAOc} at $(0,0)$, i.e. the Frechét derivative of \eqref{eq:CAOc} at $(0,0)$ yields 
$\rA_{c_*} \colon D(\rA_{c_*}) \subset \rH_0 \to \rH_0$ given by  
\begin{equation}
	\begin{aligned}
		\rA_{c_*}
		&:=
		\begin{pmatrix}
			\DeltaH + \partial_p p^2 \partial_p - c_* \cdot \nablaH & 0 \\
			0 & \Delta - c_* \cdot \nablaH
		\end{pmatrix}
		\\
		D(\rA_{c_*})
		&:= \left\{ \binom{\hvair}{\hvocn} \in \rH^2_{\bar{\sigma}}(\Omegaair)\times \rH^2_{\bar{\sigma}}(\Omegaocn) \colon V(\hvair,\hvocn) = 0, \ \text{ and } \
		\begin{matrix} 
			\partial_n \hvair|_{\Gammaairb} = \partial_n \hvair|_{\Gammaairu} = 0 \\ \partial_n \hvocn|_{\Gau} = \partial_n \hvocn|_{\Gab} = 0
		\end{matrix} \right\} .
	\end{aligned} 
\end{equation}
Here, we have used the fact that the hydrostatic Stokes operator with Neumann boundary conditions coincide with the part of the Laplacian with Neumann boundary conditions in $\rL^2_{\sigmabar}$, see \cite{GGHHK:17}, as well as
\begin{equation*}
	\divH \bar{c_* \cdot \nablaH v} = 
	c_* \cdot \nablaH \divH \bar{v} = 0
\end{equation*}
and hence 
\begin{equation*}
	\mathbb{P} (c_* \cdot \nablaH v) = c_* \cdot \nablaH v .
\end{equation*}
We denote by $\rH_1 := D(\rA_c)$ equipped with the graph norm and point out that the trace space coincides with the phase space.

\begin{lem}\label{lem:trace space}
	The complex interpolation spaces $\rH_{\frac{1}{2}} := (\rH_0,\rH_1)_{\frac{1}{2}}$ and the real interpolation space $\rH_{\frac{1}{2},2} := (\rH_0,\rH_1)_{\frac{1}{2},2}$ coincide with the phase space \eqref{eq:phase space fibre}, i.e.
	\begin{equation*}
		\rH_{\frac{1}{2}} = \rH_{\frac{1}{2},2} = \rH_\gamma . 
	\end{equation*}	
\end{lem}
\begin{proof}
	Since $\rH_0$ and $\rH_1$ are Hilbert spaces, the real and complex interpolation spaces coincide, see \cite[Theorem 4.36]{Lun:18}. For the other equality we use the fact that the bundle projection $\pi \colon \rX_\gamma \to \rH_\gamma$ 
	extends to a bounded projection $\pi_0 \colon \rX_0 \to \rH_0$. Now, \cite[Sect. 1.17.1]{Tri:78} implies
	\begin{equation*}
		\rH_\gamma = 
		\rX_{\frac{1}{2}} \cap \rH_0 
		= (\rX_0,\rX_1)_{\frac{1}{2}} \cap \rH_0 
		= (\rX_0\cap \rH_0,\rX_1\cap \rH_0)_{\frac{1}{2}}
		= (\rH_0,\rH_1)_{\frac{1}{2}} 
		= \rH_{\frac{1}{2}} . \qedhere 
	\end{equation*}
\end{proof}
 
Furthermore, the total linearization $\rA_{c_*}$ along the fibre inherits the following functional-analytic properties from the total linearization $\bar{\rA}_{c_*}$ at $(c_*,c_*)$.

\begin{cor}\label{cor:A_c mr}
	The operator $\rA_{c_*}$ admits bounded $\rH^\infty$-calculus on $\rH_0$. In particular, it has maximal regularity on $\rH_0$ and generate a compact and analytic semigroup $(T_{\rA_{c_*}}(t))_{t \geq 0}$ on $\rH_0$.  
\end{cor}

Note that the eigenfunctions in \eqref{eq:eigenfunctions} and \eqref{eq:eigenfunctions2} have zero mean. In particular, they satisfy $V(\hvair,\hvocn) = 0$.
This is in contrast to the situation of $\lambda = 0$: there we obtain that 
\begin{equation*}
	V(c_1,c_2) =  
	\frac{c_1 \cdot |\Omegaair|
		+ c_2 \cdot |\Omegaocn| }{|\Omegaair|+|\Omegaocn|} .
\end{equation*} 
Therefore, $\rA_{c_*}$ inherits the following spectral results from $\bar{\rA}_{c_*}$.

\begin{cor}\label{cor:A_c spectrum}
	The operator $\bar{\rA}_{c_*}$ has compact resolvent on $\rH_0$. Furthermore, we obtain 
	\begin{equation*}
		\sigma (\rA_{c_*}) = \sigma_p(\rA_{c_*})
		= \sigma (\bar{A}_{c_*})
	\end{equation*}
	and the eigenspaces of $\lambda \not = 0$ are generated by the functions \eqref{eq:eigenfunctions} and \eqref{eq:eigenfunctions2}.
	Therefore, they are one-dimensional unless $\log(p_a/p_s) = l \cdot h$ for some $l \in \N$, then they are two-dimensional. 
	Finally, the non-trivial kernel of $\bar{\rA}_{c_*}$ is given by 
	\begin{equation*}
		\ker(\rA_{c_*})
		= \{ (c_1,c_2)^\top \colon c_1,c_2 \in \R^2 \text{ satisfying } c_1 \cdot |\Omegaair|
		+ c_2 \cdot |\Omegaocn| = 0 \}
	\end{equation*}  
	and has dimension two. 
\end{cor}

Next, we point our that the system \eqref{eq:CAOc} inherits a strict Lyapunov functional from \eqref{eq:CAO}: its (normalized) energy is given by 
\begin{equation*}
	\rE_{c_*}(t)
	:= \rE_{c_*}(\vair,\vocn)(t) 
	:= \rE_{c_*}(\vair(t),\vocn)
	:= \frac{1}{2}\int_{\Omegaatm} |\hvair(t)|^2 + \frac{1}{2}\int_{\Omegaocn} |\hvocn(t)|^2 .
\end{equation*}
We obtain the following result.

\begin{cor}\label{cor:energy along the fibre}
	The energy $\rE_{c_*}$ is a strict Lyapunov functional. Moreover, the unique equilibrium $(0,0)$ is the unique critical point of $\rE_{c_*}$. More precisely, $(0,0)$ is the global minimizer with $\rE_{c_*}(0,0) = 0$. 
\end{cor}

\subsection{Global a-priori estimates}
\label{ssec:apriori}

\

To bridge the gap between the initial data and the neighbourhood of a steady state, we require the following a priori estimates.

\begin{lem}[Energy equality]\label{lem:energyestimates}
	Let $(\hvair,\hvocn) \in \rC([0,
	+\infty);\rH_\gamma)$ to \eqref{eq:CAOc} be the global strong solution with initial data $(\hvair_0,\hvocn_0)^\top \in \rH_\gamma$.
	Then, we have  
	\begin{align*}
		\|\hvair(t) \|^2_{\rL^2(\Omegaair)} &+ \|\hvocn(t) \|^2_{\rL^2(\Omegaocn)} + 2 \int_0^t (\| \nablaH \hvair \|^2_{\rL^2(\Omegaair)} + \| p \partial_p \hvair \|^2_{\rL^2(\Omegaair)}) \ \d s \\
		&+ 2 \int_0^t \| \nabla \hvocn \|^2_{\rL^2(\Omegaocn)} \d s + 
		\int_0^t p_s \cdot  \| [\hat{v}] \|_{\rL^3(\Gamma)}^3 \d s  
		= \| \hvair_0 \|_{\rL^2(\Omegaair)}^2 + \| \hvocn_0 \|_{\rL^2(\Omegaair)}^2 . 
	\end{align*}
\end{lem}
\begin{proof}
This is just a reformulation of the identity shown in the proof of \autoref{prop:energy lyapunov}.
\end{proof}

Note that the $c_* \cdot \nablaH$ terms do not effect the a-priori estimates since they are of lower order compared to $\DeltaH$. 
Therefore, we obtain from \cite[Proposition 6.3]{BBHZ:25} the following estimates.  

\begin{prop}[$\rL_t^\infty \rH^{1}_{\mathbf{x}}$-$\rL^2_t \rH^{2}_{\mathbf{x}}$-Estimates] \label{prop:linftyl2}
	Let $(\hvair,\hvocn) \in \rC([0,
+\infty);\rH_\gamma)$ to \eqref{eq:CAOc} be the global strong solution with initial data $(\hvair_0,\hvocn_0)^\top \in \rH_\gamma$. Then there is a 
	continuous function $B$ on~$[0,T]$, which depends on the data $\| \hvair_0 \|_{\rH^1(\Omegaair)}$, $\| \hvocn_0 \|_{\rH^1(\Omegaocn)}$ and $t \in [0,T]$, 
	such that
	\begin{equation*}
		\| \nabla \hvair(t) \|^2_{\rL^2(\Omegaair)} + \| \nabla \hvocn(t) \|^2_{\rL^2(\Omegaocn)} 
		+ \int_{\T^2} p_s |[\hat{v}(t)]|^3
		+ \int_0^t \| \Delta \hvair(s) \|^2_{\rL^2(\Omegaair)} \d s  + \int_0^t \| \Delta \hvocn(s) \|^2_{\rL^2(\Omegaocn)} \d s  
		\leq B(t). 
	\end{equation*}
\end{prop}
\begin{proof} 
Note that the boundary term $\int_{\T^2} p_s |[\hat{v}(t)]|^3$ is not included in the statement of \cite[Proposition 6.3]{BBHZ:25}. We recall the last part of step 5 of the proof of \cite[Proposition 6.3]{BBHZ:25}. It is shown that
\begin{equation*}
	\begin{aligned}
		 \partial_t \bigl(\| \nabla \vair \|^2_{\rL^2(\Omegaair)} +\| \nabla \vocn \|^2_{\rL^2(\Omegaocn)} \bigr) + \delta \bigl(  \| \Delta \vair\|^2_{\rL^2(\Omegaair)} + 
		\| \Delta \vocn\|^2_{\rL^2(\Omegaocn)} \bigr) + \frac{1}{3} \dt \biggl( \int_{\T^2} p_s |[\hat{v}]|^3 \biggr) \\
		\leq K_3(t) \bigl( \left \| \nabla \vair  \right \|^2_{\rL^2(\Omegaair)} + \| \nabla \vocn \|^2_{\rL^2(\Omegaocn)} \bigr) + K_4(t),
	\end{aligned}
\end{equation*}
with $\delta >0$, and $K_3$, $K_4$ are functions depending on $\| v_z\|_{\rL^2(\Omega)}$, $\| \nabla v_z\|_{\rL^2(\Omega)}$, $\| \vbar \|_{\rH^1(G)}$, $\| \vtilde \|_{\rL^4(\Omega)}$, $\| \pi \|_{\rL^2(G)}$.
The functions $K_3$ and $K_4$ are hence integrable with respect to time by Step 4 of the proof of \cite[Proposition 6.3]{BBHZ:25}, so Gronwall's inequality yields the existence  of a continuous function $B_2$, depending on the initial data and time, such that for $t \in [0,T]$, we have
\begin{equation*}
		\| \nabla \vair(t) \|^2_{\rL^2(\Omegaair)} + \| \nabla \vocn(t) \|^2_{\rL^2(\Omegaocn)}
		+\int_0^t \| \Delta \vair \|^2_{\rL^2(\Omegaair))} \d s+ \int_0^t\| \Delta \vocn \|^2_{\rL^2(\Omegaocn))} \d s 
		+ \int_{\T^2} p_s |[\hat{v}(t)]|^3
		\leq B_2(t). \qedhere
\end{equation*}
\end{proof} 

Further, we need the following Poincaré type estimate.

\begin{prop}[Poincaré inequality]\label{prop:poincare inequalitly}
	For $(\hvair,\hvocn)^\top \in \rH_\gamma$,
	we have the following Poincaré type inequality
	\begin{equation*}
		\| \hvair \|_{\rL^2(\Omegaair)}^2
		+ 
		\| \hvocn \|_{\rL^2(\Omegaocn)}^2
		\leq 
		C \cdot \left(
		\int_{\Omegaatm} (|\nablaH \hvair|^2 
		+ |p \partial_p \hvair|^2) + \int_{\Omegaocn} |\nabla \hvocn|^2 +\left(\int_{\T^2} p_s |[\hat{v}]|^3\right)^{\frac{2}{3}} 
		\right)
	\end{equation*}
\end{prop}
\begin{proof}
	We prove this by a standard contradiction argument. 
	Assuming the inequality is wrong, there have to exists a sequence $(\hvair_n,\hvocn_n)_{n \in \N}$ with 
	\begin{equation*} 
	\| \hvair_n \|_{\rL^2(\Omegaair)}^2
	+\| \hvocn_n \|_{\rL^2(\Omegaocn)}^2 = 1
	\end{equation*} 
	such that 
	\begin{equation*}
		\int_{\Omegaatm} (|\nablaH \hvair_n|^2 
		+ |p \partial_p \hvair_n|^2) \, \to 0, \quad \text{ and } \quad 
		\int_{\Omegaocn} |\nabla \hvocn_n|^2 \, \to 0, \quad \text{ and } \quad 
		\left(\int_{\T^2} p_s |\hvair_n - \hvocn_n|^3\right)^{\frac{2}{3}} \, \to 0 .
	\end{equation*} 
	This implies that $(\hvair_n,\hvocn_n)_{n \in \N}$ is bounded in $\rH^1(\Omegaair) \times \rH^1(\Omegaocn)$ and since the embeddings $\rH^1(\Omegaair) \hookrightarrow \rL^2(\Omegaair)$ and $\rH^1(\Omegaocn) \hookrightarrow \rL^2(\Omegaocn)$ are compact that there exists a subsequence (which we denote by abusing of notation again by $(\hvair_n,\hvocn_n)_{n \in \N}$) such that 
	\begin{equation*} 
	(\hvair_n,\hvocn_n) \to (\hvair,\hvocn) \quad \text{ in } \rL^2(\Omegaair) \times \rL^2(\Omegaocn).
	\end{equation*}
	We conclude 
	\begin{equation} 
	\| \hvair \|_{\rL^2(\Omegaair)}^2
	+ \| \hvocn \|_{\rL^2(\Omegaocn)}^2 = 1
	.\label{eq:norm 1}
	\end{equation} 
	
	\smallskip 
	
	On the other hand, we know that $\nabla \hvair_n \to 0$ in $\rL^2(\Omegaair)$ and $ \vocn_n \to 0$ in $\rL^2(\Omegaair)$
	which in particular implies 
	\begin{equation*}
	(\nabla \hvair_n,\nabla \hvocn_n) \to (0,0)
	\quad \text{ on } \mathcal{D}'(\Omegaair) \times \mathcal{D}'(\Omegaocn).
	\end{equation*}
	Hence, the distributional gradients of $\hvair$ and $\hvocn$ vanish, which implies, since $\Omegaair$ and $\Omegaocn$ are connected, that $\hvair = c_1$ and $\hvocn = c_2$ for constants $c_1,c_2 \in \R^2$. 
	Finally, $\hvair_n - \hvocn_n \to 0$ in $\rL^3(\T^2)$ implies $\hvair_n - \hvocn_n \to 0$ on $\mathcal{D}'(\T^2)$ and therefore
	\begin{equation*}
		c_1 = \hvair = \hvocn = c_2 \quad \text{ on } \mathcal{D}'(\T^2) 
	\end{equation*}
	and hence $c_1 = c_2$. 
	Since we are in $\rH_\gamma$, the constants have to satisfy
	\begin{equation*}
		c_1 \cdot |\Omegaair|
		+ c_2 \cdot |\Omegaocn| = 0 .
	\end{equation*}
	Combining both identities yield
	\begin{equation*}
		c_1 \cdot (|\Omegaair| + |\Omegaocn|)
		= c_1 \cdot |\Omegaair|
		+ c_2 \cdot |\Omegaocn| = 0 
	\end{equation*}
	and therefore $c_1 = c_2 = 0$, which contradicts \eqref{eq:norm 1}. 
\end{proof}

\begin{rem}
	Note that $(\hvair,\hvocn)^\top$ satisfy  $\int_\Omegaair \hvair + \int_\Omegaocn \hvocn = 0$ but they are not necessarily mean-free by themself. Therefore, the boundary term on the right-hand side is necessary. Indeed, consider $\hvair = c_1$ and $\hvocn = c_2$ with $c_1 \cdot |\Omegaair| = - c_2 |\Omegaocn|$. Then, the left-hand side is $|c_1|^2 \cdot |\Omegaair|^2 + |c_2|^2 \cdot |\Omegaocn|^2 > 0$, whereas the first-order terms of the right-hand side vanish. 
\end{rem}

Finally, using an argument from \cite{HK:16} -- see also \cite{Bin:25} for a discussion of it -- we conclude the following convergence result.  

\begin{cor}\label{cor:dissipation liminf}
	Let $(\hvair,\hvocn) \in \rC([0,
	+\infty);\rH_\gamma)$ to \eqref{eq:CAOc} be the global strong solution with initial data $(\hvair_0,\hvocn_0)^\top \in \rH_\gamma$.
	Its total energy $\rE_{c_*}$
	and 
	its dissipation $\mathrm{D}$ satisfy 
	\begin{align*}
		\lim_{t \to + \infty} \rE_{c_*}(t) &= 0, \\
		\liminf_{t \to + \infty} \mathrm{D}(t) &= 0 .
	\end{align*} 	
\end{cor}
\begin{proof}
	We first show the asymptotic behaviour of the dissipation: 
	From \autoref{lem:energyestimates} and \autoref{prop:linftyl2} we know that 
	\begin{equation*}
		[0,+\infty) \to [0,+\infty) \colon t \mapsto 
		\mathrm{D}(t)
	\end{equation*}
	is $1$-integrable and continuous. This implies the claim. 
	
	\smallskip 
	
	The regularity of the solution implies that 
	\begin{equation*}
		[0,+\infty) \to [0,+\infty) \colon t \mapsto \rE_{c_*}(t)
	\end{equation*}
	is continuous. Moreover, we know from \autoref{cor:energy along the fibre} that $\rE_{c_*}$ is positive, $\rE_{c_*}(t) \geq 0$, and monotonically decreasing, $\partial_t \rE_{c_*}(t) = - \mathrm{D}(t) \leq 0$. Hence, the limit $\lim_{t \to + \infty} \rE_{c_*}(t)$ exists.  
	Now, \autoref{prop:poincare inequalitly} implies
	\begin{equation*}
		\rE_{c_*}(t) \leq C \cdot ( \mathrm{D}(t) + \mathrm{D}(t)^{\frac{2}{3}} ) 
	\end{equation*}
	pointwise in time, and the desired asymptotic behaviour of $\rE_{c_*}$ follows from the one of $\mathrm{D}$.
\end{proof}

\subsection{Trapping region}
\label{ssec:trapping}

\

In this subsection, we construct a neighbourhood of the equilibrium $(c_*,c_*)$ with the constant $c_* = \mathrm{V}(\vair_0,\vocn_0)$ in the phase space $X_\gamma$ in which the global solution remains trapped. While this follows directly from the generalized principle of linearized stability for normally stable systems like the primitive equations, the present situation is more complicated. 
Instead, we rely on the a priori estimates for the system, which allow us to control the energy and dissipation and thereby confine the solution to this neighbourhood.  

\smallskip 

We first need the following relation between the $\rL^2$-norms of the Hessians and the Laplacians.

\begin{lem}[Hessian-Laplacian identity]\label{lem:bochner formula}
	For $(\hvair,\hvocn)^\top \in \rH^2(\Omegaair) \times \rH^2(\Omegaocn)$ satisfying the boundary conditions \eqref{eq:wind bc} we have the following identity
	\begin{equation}
		\| (D^a)^2 \hvair \|_{\rL^2(\Omegaair)}^2
		+
		\| D^2 \hvocn \|_{\rL^2(\Omegaair)}^2
		= \| \Delta^a \hvair \|_{\rL^2(\Omegaair)}^2
		+ \| \Delta \hvocn \|_{\rL^2(\Omegaocn)}^2
		+ 2 p_s \int_{\T^2} | [\hat{v}] | \cdot \bigl(| [\nablaH \hat{v}] |^2+( \nablaH |[\hat{v}]| )^2\bigr),
		\label{eq:bochner}
	\end{equation}
	where $D^a := (\partial_x,\partial_y,p^2 \partial_z)^\top$.
\end{lem}
\begin{proof}
	First, we let make sure that the last term on the right-hand side is well-defined assuming the desired regularity. 
	Indeed, using the embedding
	\begin{equation*}
		\rH^{\frac{3}{2}}(\T^2) \hookrightarrow \rL^\infty(\T^2),
	\end{equation*}
	Hölder inequality and the trace theorem, we obtain
	\begin{equation}
		\begin{aligned}
			\int_{\T^2} | [\hat{v}] | \cdot | \nablaH (\hvair-\hvocn) |^2 
			&\leq \| [\hat{v}] \|_{\rL^\infty(\T^2)} \cdot \| [\hat{v}] \|_{\rH^1(\T^2)}^2 \\
			&\leq C \cdot \| [\hat{v}] \|_{\rH^{\frac{3}{2}}(\T^2)} \cdot \| [\hat{v}] \|_{\rH^1(\T^2)}^2\\
			&\leq C \cdot \bigl( \| \hvair \|_{\rH^{2}(\Omegaair)} + 
			\| \hvocn \|_{\rH^{2}(\Omegaocn)} \bigr)
			\cdot \bigl( \| \hvair \|_{\rH^{\frac{3}{2}}(\Omegaair)}^2 + 
			\| \hvocn \|_{\rH^{\frac{3}{2}}(\Omegaocn)}^2 \bigr).
		\end{aligned}
		\label{eq:estimate boundary term bochner}
	\end{equation} 
	For the other term note that $(\nablaH |\hvair-\hvocn| )^2 \leq | \nablaH (\hvair-\hvocn) |^2$ and hence it follows from the previous estimate.
	Now, integration by parts yields
	\begin{align*}
		\int_{\Omegaocn} (\Delta \hvocn) \cdot (\Delta \hvocn)
		&= \int_{\Omegaocn} \partial_i \partial^i \hvocn_k \partial_j \partial^j \hvocn_k \\
		&= - \int_{\Omegaocn} 
		\partial^i \hvocn_k \partial_i \partial_j \partial^j \hvocn_k 
		+ \int_{\Gau} \partial^i \hvocn_k \partial_j \partial^j \hvocn_k n_i \\
		&= \int_{\Omegaocn} 
		\partial_j \partial^i \hvocn_k \partial_i \partial^j \hvocn_k 
		+ \int_{\Gau} \partial^i \hvocn_k \partial_j \partial^j \hvocn_k n_i
		- \int_{\Gau} \partial^i \hvocn_k \partial_i \partial^j \hvocn_k n_j
		\\
		&= \int_{\Omegaocn} 
		\partial_j \partial^i \hvocn_k \partial_i \partial^j \hvocn_k 
		+ \int_{\Gau} \partial^i \hvocn_k \partial_\eta \partial^\eta \hvocn_k n_i
		- \int_{\Gau} \partial^\eta \hvocn_k  \partial_\eta \partial^j \hvocn_k n_j \\
		&= \int_{\Omegaocn} 
		\partial_j \partial^i \hvocn_k \partial_i \partial^j \hvocn_k 
		- 2 \int_{\Gau} \partial^\eta \hvocn_k  \partial_\eta \partial^j \hvocn_k n_j \\
		&= \int_{\Omegaocn} |D^2 \hvocn|^2 
		- 2 \int_{\T^2} (\nablaH \hvocn) \cdot  \nablaH (\partial_z \hvocn) ,
	\end{align*}
	where we have used $(\partial_z \hvocn)|_{\Gab} = 0$. Here indices $i,j,k \in \{x,y,z\}$ and $\eta \in \{x,y\}$.
	Using the non-linear boundary conditions, we obtain
	\begin{align*}
		\int_{\Omegaocn} (\Delta \hvocn) \cdot (\Delta \hvocn)
		&= \int_{\Omegaocn} |D^2 \hvocn|^2 
		- 2 \int_{\T^2} (\nablaH \hvocn) \cdot  \nablaH (\partial_z \hvocn) \\
		&= \int_{\Omegaocn} |D^2 \hvocn|^2 
		- 2 p_s \int_{\T^2} (\nablaH \hvocn) \cdot \nablaH (|\hvair-\hvocn| \cdot (\hvair-\hvocn)) . 
	\end{align*}
	By a similar calculation, we obtain
	\begin{align*}
		\int_{\Omegaair} (\Delta^a \hvair)\cdot  (\Delta^a \hvair)
		&= \int_{\Omegaair} |(D^a)^2 \hvair|^2 
		+ 2 p_s \int_{\T^2} (\nablaH \hvair) \cdot \nablaH (|\hvair-\hvocn| \cdot (\hvair-\hvocn)) .
	\end{align*}
	Adding both equalities yields 
	\begin{align*}
		\int_{\Omegaair} |\Delta^a \hvair|^2
		&+
		\int_{\Omegaocn} |\Delta \hvocn|^2
		= \int_{\Omegaair} |(D^a)^2 \hvair|^2 
		+ \int_{\Omegaocn} |D^2 \hvocn|^2
		+ 2 p_s \int_{\T^2} \nablaH (\hvair-\hvocn) \cdot \nablaH (|\hvair-\hvocn| \cdot (\hvair-\hvocn))\\
		&= \int_{\Omegaair} |(D^a)^2 \hvair|^2 
		+ \int_{\Omegaocn} |D^2 \hvocn|^2
		+ 2 p_s \int_{\T^2} | [\hat{v}] | \cdot \bigl(| \nablaH (\hvair-\hvocn) |^2+( \nablaH |\hvair-\hvocn| )^2\bigr) .
		\qedhere 
	\end{align*}
\end{proof}

Note that the last term on the right-hand side can be controlled by the $\rH^2$-norms but the powers do not match. This is an artefact of the non-linearity of the coupling conditions. Nevertheless, we obtain equivalence of the $\rH^2$-norm and the graph norm of the Laplacians provided that the $\rH^1$-norms are sufficiently small.

\begin{cor}[Elliptic estimate]\label{cor:elliptic estimate}
	Let $(\hvair,\hvocn)^\top \in \rH^2(\Omegaair) \times \rH^2(\Omegaocn)$ with boundary conditions \eqref{eq:wind bc}. 
	Assume that 
	\begin{equation*}
		\| \hvair \|_{\rH^1(\Omegaair)} + \| \hvocn \|_{\rH^1(\Omegaocn)} <\varepsilon 
	\end{equation*}
	with $\varepsilon > 0$ sufficiently small. 
	Then
	we obtain the a-priori estimate 
	\begin{equation*}
		\| \hvair \|_{\rH^2(\Omegaair)}^2
		+ \| \hvocn \|_{\rH^2(\Omegaocn)}^2  
		\leq C \cdot \bigl( 
		\| \Delta^a \hvair \|_{\rL^2(\Omegaair)}^2
		+ \| \hvair \|_{\rL^2(\Omegaair)}^2
		+ \| \Delta \hvocn \|_{\rL^2(\Omegaocn)}^2
		+ \| \hvocn \|_{\rL^2(\Omegaocn)}^2 \bigr) .   
	\end{equation*}
\end{cor}
\begin{proof}
	Using the estimates from \eqref{eq:estimate boundary term bochner} and interpolation we obtain
	\begin{align*}
		\int_{\T^2} | [\hat{v}] | \cdot | \nablaH (\hvair-\hvocn) |^2 
		&\leq C \cdot \bigl( \| \hvair \|_{\rH^{2}(\Omegaair)} + 
		\| \hvocn \|_{\rH^{2}(\Omegaocn)} \bigr)
		\cdot \bigl( \| \hvair \|_{\rH^{\frac{3}{2}}(\Omegaair)}^2 + 
		\| \hvocn \|_{\rH^{\frac{3}{2}}(\Omegaocn)}^2 \bigr) \\
		&\leq C \cdot \bigl( \| \hvair \|_{\rH^{1}(\Omegaair)} + 
		\| \hvocn \|_{\rH^{1}(\Omegaocn)} \bigr)
		\cdot \bigl( \| \hvair \|_{\rH^{2}(\Omegaair)}^2 + 
		\| \hvocn \|_{\rH^{2}(\Omegaocn)}^2 \bigr) .
	\end{align*}
	If $\varepsilon > 0$ is sufficiently large, we obtain
	\begin{equation*}
		2 p_s \int_{\T^2} | [\hat{v}] | \cdot | \nablaH (\hvair-\hvocn) |^2 
		\leq \frac{1}{4} \cdot \bigl( \| \hvair \|_{\rH^{2}(\Omegaair)}^2 + 
		\| \hvocn \|_{\rH^{2}(\Omegaocn)}^2 \bigr) .
	\end{equation*}
	Finally, the other term follows from $(\nablaH |\hvair-\hvocn| )^2 \leq | \nablaH (\hvair-\hvocn) |^2$ and we conclude the claim by absorbing into the left-hand side of \eqref{eq:bochner}. 
\end{proof}

Note that the $\| \hvair \|_{\rL^2(\Omegaair)}^2$ and $\| \hvocn \|_{\rL^2(\Omegaocn)}^2$ cannot be dropped due to constants. Even, assuming $V(\hvair,\hvocn) = 0$ does not yield to a stronger estimate, since there are still constants left. 
In order, to handle them, we need the non-linear boundary conditions \eqref{eq:wind bc}. We obtain the following non-linear elliptic estimate. 

\begin{lem}[Non-linear elliptic estimate]\label{lem:elliptic estimate 2}
	Let $(\hvair,\hvocn)^\top \in \rH^2(\Omegaair) \times \rH^2(\Omegaocn)$ with boundary conditions \eqref{eq:wind bc} and $V(\hvair,\hvocn)  = 0$. 
	Assume that 
	\begin{equation*}
		\| \hvair \|_{\rH^1(\Omegaair)}^2 + \| \hvocn \|_{\rH^1(\Omegaocn)}^2 + \| [\hat{v}] \|_{\rL^3(\T^2)}^3 < \varepsilon 
	\end{equation*}
	with $\varepsilon > 0$ sufficiently small. 
	Then
	we obtain the a-priori estimate 
	\begin{equation*}
		\| \Delta^a \hvair \|_{\rL^2(\Omegaair)}^2
		+ \| \Delta \hvocn \|_{\rL^2(\Omegaocn)}^2  
		\geq c_0 \cdot ( \| \nabla^a \hvair \|_{\rL^2(\Omegaair)}^2 + \| \nabla \hvocn \|_{\rL^2(\Omegaocn)}^2 + \| [\hat{v}] \|_{\rL^3(\T^2)}^3 )^{\frac{4}{3}} .
	\end{equation*}
\end{lem}
\begin{proof}
	Multiplying by $\hvair$ and $\hvocn$, integrating by parts, Cauchy-Schwarz inequality and Young's inequality yields
	\begin{align*}  
		\ \| \nabla^a \hvair \|_{\rL^2(\Omegaair)}^2 + \| \nabla \hvocn \|_{\rL^2(\Omegaocn)}^2 + p_s \cdot \| [\hat{v}] \|_{\rL^3(\T^2)}^3
		= -\int_{\Omegaair} \Delta^a \hvair \hvair -\int_{\Omegaocn} \Delta \hvocn \hvocn \\
		\leq \, C \cdot 
		(\| \Delta^a \hvair \|_{\rL^2(\Omegaair)}
		+ \| \Delta \hvocn \|_{\rL^2(\Omegaocn)})^{\frac{3}{2}}
		+ \varepsilon \cdot 
		(\| \hvair \|_{\rL^2(\Omegaair)}
		+ \| \hvocn \|_{\rL^2(\Omegaocn)})^3  
	\end{align*}
	for every $\varepsilon > 0$.
	Now, the Poincaré inequality from \autoref{prop:poincare inequalitly} implies
	\begin{align*}
		\| \hvair \|_{\rL^2(\Omegaair)}^3
		+ \| \hvocn \|_{\rL^2(\Omegaocn)}^3
		&\leq 
		(\| \hvair \|_{\rL^2(\Omegaair)}^2
		+ \| \hvocn \|_{\rL^2(\Omegaocn)}^2)^{\frac{3}{2}} \\
		&\leq 
		C \cdot \left(
		\| \nabla^a \hvair \|_{\rL^2(\Omegaair)}^2 +
		\| \nabla \hvocn \|_{\rL^2(\Omegaair)}^2 +
		p_s \cdot \| [\hat{v}] \|_{\rL^3(\T^2)}^2
		\right)^{\frac{3}{2}} \\
		&\leq 
		C \cdot \left(
		\| \nabla^a \hvair \|_{\rL^2(\Omegaair)}^3 +
		\| \nabla \hvocn \|_{\rL^2(\Omegaair)}^3 +
		p_s \cdot	\| [\hat{v}] \|_{\rL^3(\T^2)}^3
		\right) \\
		&\leq 
		C \cdot \left(
		\| \nabla^a \hvair \|_{\rL^2(\Omegaair)}^2 +
		\| \nabla \hvocn \|_{\rL^2(\Omegaair)}^2 +
		p_s \cdot \| [\hat{v}] \|_{\rL^3(\T^2)}^3
		\right),
	\end{align*}
	where we have used in the last line that $\| \nabla^a \hvair \|_{\rL^2(\Omegaair)} \leq 1$ and $\| \nabla \hvocn \|_{\rL^2(\Omegaocn)} \leq 1$.
	By choosing $\varepsilon> 0$ sufficiently small, we can absorb the second term on the right-hand side and obtain
	\begin{align*}
		\ \| \nabla^a \hvair \|_{\rL^2(\Omegaair)}^2 + \| \nabla \hvocn \|_{\rL^2(\Omegaocn)}^2 + p_s \cdot  \| [\hat{v}] \|_{\rL^3(\T^2)}^3 
		&\leq \, C \cdot 
		(\| \Delta^a \hvair \|_{\rL^2(\Omegaair)}
		+ \| \Delta \hvocn \|_{\rL^2(\Omegaocn)})^{\frac{3}{2}} \\
		&\leq \, C \cdot 
		(\| \Delta^a \hvair \|_{\rL^2(\Omegaair)}^2
		+ \| \Delta \hvocn \|_{\rL^2(\Omegaocn)}^2)^{\frac{3}{4}} .
		\qedhere 
	\end{align*}
\end{proof}

\begin{lem}[Differential inequality of the Dissipation]\label{lem:differential inequality dissipation}
	Let $(\hvair,\hvocn) \in \rC([0,
	+\infty);\rH_\gamma)$ to \eqref{eq:CAOc} be the global strong solution with initial data $(\hvair_0,\hvocn_0)^\top \in \rH_\gamma$.
	Then the dissipation functional
	\begin{equation*}
		\mathrm{D}(t)
		:= \| \nabla^a \hvair(t) \|_{\rL^2(\Omegaair)}^2
		+ 
		\| \nabla \hvocn(t) \|_{\rL^2(\Omegaocn)}^2 
		+ \frac{p_s}{3} \int_{\T^2} |[\hat{v}](t)|^3
	\end{equation*}
	satisfies
	\begin{align*}
		&\frac{1}{2} \partial_t \mathrm{D}(t)
		+ ( \| \Delta^a \hvair \|_{\rL^2(\Omegaair)}^2  +\| \Delta \hvocn \|_{\rL^2(\Omegaocn)}^2  ) \\
		\leq \ 
		&C \cdot (\|\hvair \|_{\rH^1(\Omegaair)}^2  +\| \hvocn \|_{\rH^1(\Omegaocn)}^2 ) \cdot ( \| \Delta^a \hvair \|_{\rL^2(\Omegaair)}^2  +\| \Delta \hvocn \|_{\rL^2(\Omegaocn)}^2  + \| \hvair \|_{\rL^2(\Omegaair)}^2  +\| \hvocn \|_{\rL^2(\Omegaocn)}^2 ) .
	\end{align*}
\end{lem}
\begin{proof}
		Using integration by parts and the boundary conditions \eqref{eq:wind bc} we obtain
	\begin{align*}
		\int_{\Omegaair} \partial_t \hvair (-\Delta^a \hvair)
		+ 
		\int_{\Omegaocn} \partial_t \hvocn (-\Delta \hvocn)
		&= 
		\int_{\Omegaair} \partial_t \nabla^a \hvair \cdot \nabla^a \hvair
		- p_s^2 \int_{\T^2} \partial_p \hvair \partial_t \hvair 
		+ 
		\int_{\Omegaocn} \partial_t \nabla \hvocn\cdot \nabla \hvocn
		- \int_{\T^2} \partial_z \hvocn \partial_t \hvocn 
		\\
		&= 
		\int_{\Omegaair} \partial_t \nabla^a \hvair \cdot \nabla^a \hvair
		+ 
		\int_{\Omegaocn} \partial_t \nabla \hvocn\cdot \nabla \hvocn \\
		&+ p_s \int_{\T^2} ([\hat{v}]) |\hvair -\hvocn| \cdot \partial_t \hvair 
		- p_s \int_{\T^2} ([\hat{v}]) |\hvair -\hvocn| \cdot \partial_t \hvocn 
		\\
		&= \frac{1}{2} 
		\partial_t
		\| \nabla^a \hvair \|_{\rL^2(\Omegaair)}^2
		+ \frac{1}{2} 
		\partial_t
		\| \nabla^a \hvair \|_{\rL^2(\Omegaair)}^2
		+ 
		\frac{p_s}{3}\partial_t
		\| [\hat{v}] \|_{\rL^3(\T^2)}^3 .
	\end{align*}
	Hence, multiplying the first equation in \eqref{eq:CAOc} by $(-\Delta^a \hvair)$ and the third by $(-\Delta \hvocn)$ yields
	\begin{equation}
		\begin{aligned}
			& \quad \frac{1}{2} 
			\partial_t
			\| \nabla^a \hvair \|_{\rL^2(\Omegaair)}^2
			+ \frac{1}{2} 
			\partial_t
			\| \nabla^a \hvair \|_{\rL^2(\Omegaair)}^2
			+ 
			\frac{p_s}{3}\partial_t
			\| [\hat{v}] \|_{\rL^3(\T^2)}^3
			+ \| \Delta^a \hvair \|_{\rL^2(\Omegaair)}^2
			+ \| \Delta \hvocn \|_{\rL^2(\Omegaocn)}^2 \\
			&= \int_{\Omegaair} c_* \cdot \nablaH \hvair \cdot \Delta^a \hvair
			+ \int_{\Omegaair} \hvair \cdot \nablaH \hvair \cdot \Delta^a \hvair
			+ \int_{\Omegaair} \hat{\omega} \partial_p \hvair \cdot \Delta^a \hvair
			+ \int_{\Omegaair} \nablaH \hat{\Phi}_s \cdot  \Delta^a \hvair \\
			&+
			\int_{\Omegaocn} c_* \cdot \nablaH \hvocn \cdot \Delta \hvocn
			+ \int_{\Omegaocn} \hvocn \cdot \nablaH \hvocn \cdot \Delta \hvocn
			+ \int_{\Omegaocn} \hat{w} \partial_z \hvocn \cdot \Delta \hvocn
			+ \int_{\Omegaocn} \nablaH \hat{\pi}_s \cdot \Delta \hvocn .
		\end{aligned}
		\label{eq:testing}
	\end{equation}
	First, we show that there is no contribution from the the lower order terms depending on the equilibrium $(c_*,c_*)$. Indeed integrating by parts yields
	\begin{align*}
		\int_{\Omegaocn} c_* \cdot \nablaH \hvocn \cdot \Delta \hvocn
		&= 
		\int_{\Omegaocn} c_*^\eta \partial_\eta \hvair_k \partial_l \partial^l \hvocn_k = 
		- \int_{\Omegaocn} c_*^\eta \hvocn_k \partial_\eta \partial_l \partial^l \hvocn_k \\
		&= 
		- \int_{\Omegaocn} c_*^\eta \partial_l \partial^l \hvocn_k \cdot \partial_\eta \hvocn_k
		- \int_{\T^2} c_*^\eta \hvocn_k \partial_z \partial_{\eta} \hvocn_k 
		+ \int_{\T^2} c_*^\eta \partial_z \hvocn_k \partial_{\eta} \hvocn_k \\
		&= 
		- \int_{\Omegaocn} c_*^\eta \partial_l \partial^l \hvocn_k \cdot \partial_\eta \hvocn_k
		+ 2 \int_{\T^2} c_*^\eta \partial_z \hvocn_k \partial_{\eta} \hvocn_k \\
		&= 
		- 
		\int_{\Omegaocn} c_* \cdot \nablaH \hvocn \cdot \Delta \hvocn
		+ 2 \int_{\T^2} (c_* \cdot \nablaH \hvocn) \cdot (\partial_z \hvocn)
	\end{align*}
	for $\eta \in \{x,y\}$ and $k,l \in \{x,y,z\}$.
	Using the coupling boundary conditions \eqref{eq:wind bc}, this implies
	\begin{equation*}
		\int_{\Omegaocn} c_* \cdot \nablaH \hvocn \cdot \Delta \hvocn
		= p_s \int_{\T^2} (c_* \cdot \nablaH \hvocn) \cdot ([\hat{v}]) |[\hat{v}]| .
	\end{equation*}
	A similar calculation yields
	\begin{equation*}
		\int_{\Omegaair} c_* \cdot \nablaH \hvair \cdot \Delta^a \hvair
		= - p_s \int_{\T^2} (c_* \cdot \nablaH \hvair) \cdot ([\hat{v}]) |[\hat{v}]| .
	\end{equation*}
	Combining the two identities and using the Gauss theorem, we obtain
	\begin{align*}
		\int_{\Omegaair} c_* \cdot \nablaH \hvair \cdot \Delta^a \hvair
		+ \int_{\Omegaocn} c_* \cdot \nablaH \hvocn \cdot \Delta \hvocn
		&= - \frac{p_s}{3} \int_{\T^2} (c_* \cdot \nablaH)(|[\hat{v}]|^3) \\
		&= - \frac{p_s}{3} \int_{\T^2} \divH( |[\hat{v}]|^3 \, c_*) 
		= 0 .
	\end{align*}
	Next, we deal with the terms in \eqref{eq:testing} coming from the non-linearities. 
	Using the estimates from \cite[Lemma 5.1]{HK:16} and interpolation inequality we obtain for the non-linearities
	\begin{align*}
		\| \hvair \cdot \nablaH \hvair
		+ \hat{\omega} \partial_p \hvair \|_{\rL^2(\Omegaair)}^2
		\leq 
		C \cdot 
		\| \hvair \|_{\rH^{\frac{3}{2}}(\Omegaair)}^4
		\leq 
		C \cdot 
		\| \hvair \|_{\rH^1(\Omegaair)}^2 \cdot
		\| \hvair \|_{\rH^2(\Omegaair)}^2 , \\
		\| \hvocn \cdot \nablaH \hvocn
		+ \hat{w} \partial_z \hvocn \|_{\rL^2(\Omegaocn)}^2
		\leq 
		C \cdot 
		\| \hvocn \|_{\rH^{\frac{3}{2}}(\Omegaocn)}^4
		\leq 
		C \cdot 
		\| \hvocn \|_{\rH^1(\Omegaocn)}^2 \cdot
		\| \hvocn \|_{\rH^2(\Omegaocn)}^2 .
	\end{align*}
	Finally, it remains to estimate the geopotential and the ocean pressure. Since, the estimates are essentially the same, we only carry out the later one. We start by recalling that the oceanic pressure is given by
	\begin{equation*}
		\nablaH \pi_s 
		= (\nablaH \DeltaH^{-1} \divH) ( \overline{ \hvocn \cdot \nablaH \hvocn + w \partial_z \hvocn } - \partial_z \hvocn|_{\Gau} ),
		=: \nablaH \pi_{\mathrm{NL}} 
		+ \nablaH \pi_{\mathrm{CL}},
	\end{equation*}
	where subscript $\mathrm{NL}$ denotes the part arising from the non-linearity and $\mathrm{CL}$ the part arising from the coupling boundary conditions \eqref{eq:wind bc}. 
	Calderon-Zygmund estimates and $\| \overline{\hvocn} \|_{\rL^2(\T^2)} \leq \| \hvocn \|_{\rL^2(\Omegaocn)}$ yields that 
	\begin{equation*}
		\| \nablaH \pi_{\mathrm{NL}} \|_{\rL^2(\T^2)}^2
		\leq C \cdot \| \overline{ \hvocn \cdot \nablaH \hvocn + w \partial_z \hvocn } \|_{\rL^2(\T^2)}^2
		\leq C \cdot \| \hvocn \cdot \nablaH \hvocn + w \partial_z \hvocn \|_{\rL^2(\Omegaocn)}^2
	\end{equation*}
	and the desired estimate follows from the last step. For the coupling part, we apply Calderon-Zygmund estimates and use the boundary conditions \eqref{eq:wind bc} to obtain
	\begin{equation*}
		\| \nablaH \pi_{\mathrm{CL}} \|_{\rL^2(\T^2)}^2
		\leq 
		C \cdot \| \partial_z \hvocn|_{\Gau} \|_{\rL^2(\T^2)}^2
		= C \cdot \| [\hat{v}] \|_{\rL^4(\T^2)}^4 .
	\end{equation*}
	Now, the embedding
	\begin{equation*}
		\rH^{\frac{1}{2}}(\T^2) \hookrightarrow \rL^4(\T^2)
	\end{equation*} 
	and the trace inequality imply
	\begin{align*}
		\| \nablaH \pi_{\mathrm{CL}} \|_{\rL^2(\T^2)}^2
		\leq 
		C \cdot \| [\hat{v}] \|_{\rL^4(\T^2)}^4 
		\leq C \cdot \| [\hat{v}] \|_{\rH^{\frac{1}{2}}(\T^2)}^4  
		\leq C \cdot ( \| \hvair \|_{\rH^1(\Omegaair)}^4 + \| \hvocn \|_{\rH^1(\Omegaocn)}^4 ) .
	\end{align*}
	Combining all estimates an absorbing argument and Young's inequality yield the claim. 
\end{proof}

\begin{prop}[Construction of the trapping region]\label{prop:trapping}
	Let $(\hvair,\hvocn) \in \rC([0,
	+\infty);\rH_\gamma)$ to \eqref{eq:CAOc} be the global strong solution with initial data $(\hvair_0,\hvocn_0)^\top \in \rH_\gamma$.
	For $\varepsilon > 0$ sufficiently small there exists $t_0 > 0$ such that 
		\begin{equation*}
			\| \hvair(t) \|_{\rH^1(\Omegaair)}^2
			+ 
			\| \hvocn(t) \|_{\rH^1(\Omegaocn)}^2 
			+\int_{\T^2} p_s |[\hat{v}(t)]|^3
			< \varepsilon \qquad \text{ for all } t > t_0 .  
		\end{equation*}
\end{prop}
\begin{proof}
	For every $T >0$ we have $(\hvair,\hvocn) \in \rC([0,T],\rH_\gamma)$.
	We consider the quantity 
	\begin{equation*}
		\mathrm{Q}(t)
		:= \rE_{c_*}(t) + \mathrm{D}(t)
		= \| \hvair(t) \|_{\rH^1(\Omegaair)}^2
		+ 
		\| \hvocn(t) \|_{\rH^1(\Omegaocn)}^2 
		+\int_{\T^2} p_s |[\hat{v}(t)]|^3 .
	\end{equation*}
	For $0 < \delta < \varepsilon$, which will be determined later, there exists by \autoref{cor:dissipation liminf} a time $t_0 > 0$ such that $\mathrm{Q}(t_0) < \delta$. We consider  
	the set $S$ where the quantity $\mathrm{Q}$ is trapped
	\begin{equation*}
		S := \{ T > t_0 \colon \mathrm{Q}(t) < \varepsilon \text{ for all } t \in [t_0,T] \} 
	\end{equation*}
	and its maximal point in time
	\begin{equation*}
		T_* := \sup S .
	\end{equation*}
	Since $0 < \delta < \varepsilon$ continuity of the quantity $\mathrm{Q}$, see \autoref{lem:energyestimates} and \autoref{prop:linftyl2}, implies $T_* > t_0$. 
	Now, we choose $\varepsilon> 0$ small enough to guarantee the assumptions of \autoref{cor:elliptic estimate} and \autoref{lem:elliptic estimate 2}.
	Using \autoref{lem:differential inequality dissipation} we obtain
	\begin{equation*}
		\frac{1}{2} \partial_t \mathrm{D}(t)
		+ ( \| \Delta^a \hvair \|_{\rL^2(\Omegaair)}^2  +\| \Delta \hvocn \|_{\rL^2(\Omegaocn)}^2  )
		\leq 
		C \cdot \mathrm{Q}(t) \cdot ( \| \Delta^a \hvair \|_{\rL^2(\Omegaair)}^2  +\| \Delta \hvocn \|_{\rL^2(\Omegaocn)}^2  + \rE_{\mathrm{c}_*}(t) ) .
	\end{equation*}
	Choosing $\varepsilon >0$ sufficiently small, we have $C \cdot \mathrm{Q}(t) < C \cdot \varepsilon < \frac{1}{2}$ and an absorbing argument and \autoref{lem:energyestimates} yield
	\begin{equation*}
		\partial_t \mathrm{D}(t)
		+ ( \| \Delta^a \hvair \|_{\rL^2(\Omegaair)}^2  +\| \Delta \hvocn \|_{\rL^2(\Omegaocn)}^2  )
		\leq 
		C \cdot \varepsilon \cdot \rE_{\mathrm{c}_*}(t) 
		\leq 
		C \cdot \varepsilon \cdot \rE_{\mathrm{c}_*}(t_0) 
		\leq C \cdot \varepsilon \cdot \delta . 
	\end{equation*}
	Now, applying \autoref{lem:elliptic estimate 2} we obtain
	\begin{equation*}
		\partial_t \mathrm{D}(t)
		+ c_0 \cdot \mathrm{D}(t)^{\frac{4}{3}}
		\leq C \cdot \varepsilon \cdot \delta .
	\end{equation*}
	By \autoref{cor:analyticity} we know that the dissipation $\mathrm{D}$ is $\rC^1$ in time, and we conclude  
	\begin{equation}
		\mathrm{D}(t)
		\leq \max\left\{ \mathrm{D}(t_0),\biggl( \frac{C \cdot \varepsilon \cdot \delta}{c_0}\biggr)^{\frac{3}{4}} \right\} .
		\label{eq:D2}
	\end{equation}
	By choosing $\delta > 0$ accordingly we obtain from \eqref{eq:D2} and $\rE_{c_*}(t) \leq \rE_{c_*}(t_0) < \delta$ that 
	\begin{equation}
		\mathrm{Q}(t)
		\leq \delta + \max\left\{ \mathrm{D}(t_0),\biggl( \frac{C \cdot \varepsilon \cdot \delta}{c_0}\biggr)^{\frac{3}{4}} \right\}
		\leq \frac{\varepsilon}{2} . \label{eq:Q}
	\end{equation}
	Finally, we assume by contradiction $T_* < + \infty$. By continuity of the quantity $\mathrm{Q}$ and the definition of the set $S$ we obtain $\mathrm{Q}(T_*) = \varepsilon$. On the other hand continuity of the quantity $\mathrm{Q}$ and the inequality \eqref{eq:Q} imply $\mathrm{Q}(T_*) \leq \frac{\varepsilon}{2}$ which is a contradiction. Therefore $T_* = + \infty$ and the claim follows. 
\end{proof}

\subsection{Convergence to equilibria and optimal convergence rate}
\label{ssec:convergence rate}

\

After having established that the global strong solution  $(\hvair,\hvocn) \in \rC([0,
+\infty);\rH_\gamma)$ to \eqref{eq:CAOc} can be trapped in a neighbourhood in $X_\gamma$ of the equilibrium $(c_*,c_*)$ with the constant $c_* = \mathrm{V}(\vair_0,\vocn_0)$, we now aim to show that it actually converges to the equilibrium $(c_*,c_*)$ in the $\rH^1$-norm as $t \to + \infty$ and determine the rate of convergence. 

\smallskip

We start with the following estimates of the energy, which is a direct consequence of \autoref{prop:energy lyapunov} and \autoref{prop:poincare inequalitly}.

\begin{lem}[Energy inequalities near equilibrium]\label{lem:differential inequality}
	Consider initial data $(\hvair_0,\hvocn_0)^\top \in \rH_\gamma$
	with 
	\begin{equation*}
		\int_{\Omegaatm} (|\nablaH \vair|^2 
		+ |p \partial_p \vair|^2) + \int_{\Omegaocn} |\nabla \vocn|^2 +\int_{\T^2} p_s |[v]|^3
			< \varepsilon \quad \text{ for all } t \geq t_0,
	\end{equation*} 
	with $\varepsilon >0$ from \autoref{prop:trapping},
	then there exists a constant $C > 0$ such that the energy satisfies the different inequality
	\begin{equation*}
		\partial_t \rE_{c_*}(t)
		\leq - C \cdot \rE(t)^{\frac{3}{2}} \quad \text{ for all } t \geq t_0 .   
	\end{equation*} 	
	This implies
	\begin{equation*}
		\rE_{c_*}(t)
		\leq \frac{C}{(\rE_{c_*}(t_0)^{-\frac{1}{2}} + t)^2}  \quad \text{ for all } t \geq t_0 .
	\end{equation*}
\end{lem}
\begin{proof}
	We recall from the proof of \autoref{prop:energy lyapunov} that 
	\begin{equation*}
		- \partial_t \rE_{c_*}(t)
		= \int_{\Omegaatm} (|\nablaH \vair|^2 
		+ |p \partial_p \vair|^2) + \int_{\Omegaocn} |\nabla \vocn|^2 +\int_{\T^2} p_s |[v]|^3 .
	\end{equation*}
	Applying the basic algebraic inequality $a+b^{\frac{2}{3}} \leq 2 (a+b)^{\frac{2}{3}}$ for $a, b \geq 0$ with $a+b \leq 1$, to the quantities $a := \int_{\Omegaatm} (|\nablaH \vair|^2 
	+ |p \partial_p \vair|^2) + \int_{\Omegaocn} |\nabla \vocn|^2$ and $b := \int_{\T^2} p_s |[v]|^3$ we obtain from \autoref{prop:poincare inequalitly} the estimate
	\begin{align*}
		\quad &\int_{\Omegaatm} (|\nablaH \vair|^2 
		+ |p \partial_p \vair|^2) + \int_{\Omegaocn} |\nabla \vocn|^2 +\int_{\T^2} p_s |[v]|^3 \\
		\geq 
		\ & C \cdot \left( \int_{\Omegaatm} (|\nablaH \vair|^2 
		+ |p \partial_p \vair|^2) + \int_{\Omegaocn} |\nabla \vocn|^2 +\left(\int_{\T^2} p_s |[v]|^3 \right)^{\frac{2}{3}}
		\right)^{\frac{3}{2}}
		\geq C \cdot \rE(t)^{\frac{3}{2}} .
	\end{align*}
	Hence,
	\begin{equation*}
		- \partial_t \rE_{c_*}(t)
		\geq C \cdot \rE(t)^{\frac{3}{2}} .
	\end{equation*}
	Note that, the borderline case ordinary differential equation is
	\begin{equation*}
		\partial_t \rE_{c_*}(t)
		= - C \cdot \rE(t)^{\frac{3}{2}} \quad \text{ for all } t \geq t_0 ,
	\end{equation*}
	and its solution is given by
	\begin{equation*}
		\rE_{c_*}(t)
		= \frac{C}{(\rE_{c_*}(t_0)^{-\frac{1}{2}} + t)^2}  \quad \text{ for all } t \geq t_0 .
	\end{equation*}
	Finally, the claim follows from the comparison principle of ordinary different equations. 
\end{proof}

We conclude the following result about the convergence rate. 

\begin{lem}[Polynomial convergence rate in $\rL^2$]\label{lem:algebraic convergence rate L^2}
	Let $(\hvair,\hvocn) \in \rC([0,
+\infty);\rH_\gamma)$ to \eqref{eq:CAOc} be the global strong solution with initial data $(\hvair_0,\hvocn_0)^\top \in \rH_\gamma$.
Then the total energy satisfies 
	\begin{equation*}
		\| \hvair(t) \|_{\rL^2(\Omegaair)}^2
		+ 
		\| \hvocn(t) \|_{\rL^2(\Omegaocn)}^2 
		\leq \frac{C}{(1+t)^2} \qquad \text{ for all } t \geq 0 .  
	\end{equation*}
\end{lem}
\begin{proof}
	Assume $(\hvair_0,\hvocn_0) \not = (0,0)$, since else-wise the result is trivial. 
	
	\smallskip 
	
	By \autoref{prop:trapping} we may choose $t_0 \in [0,+\infty)$ such that
	\begin{equation*}
		\int_{\Omegaatm} (|\nablaH \vair|^2 
		+ |p \partial_p \vair|^2) + \int_{\Omegaocn} |\nabla \vocn|^2 +\int_{\T^2} p_s |[v]|^3
		< \varepsilon \quad \text{ for all } t \geq t_0 .
	\end{equation*} 
	Due to \autoref{cor:not finite time} we know that $\rE_{c_*}(t_0) \not = 0$.
	Now, it follows from \autoref{lem:differential inequality} that
	\begin{align*}
		\| \hvair(t) \|_{\rL^2(\Omegaair)}^2
		+ 
		\| \hvocn(t) \|_{\rL^2(\Omegaocn)}^2
		= 2 \cdot 
		\rE_{c_*}(t)
		\leq \frac{C}{(\rE_{c_*}(t_0)^{-\frac{1}{2}} + t)^2} 
		\leq \frac{\tilde{C}}{(1+t)^2}
		 \quad \text{ for all } t \geq t_0 .
	\end{align*}
	On the other hand, for $t \in [0,t_0]$ we know $\| \hvair(t) \|_{\rL^2(\Omegaair)}^2
	 + 
	 \| \hvocn(t) \|_{\rL^2(\Omegaocn)}^2 \leq \| \hvair_0 \|_{\rL^2(\Omegaair)}^2
	 + 
	 \| \hvocn_0 \|_{\rL^2(\Omegaocn)}^2$ by \autoref{lem:energyestimates}. 
	 Combining both yields the claim. 
\end{proof}

%
Finally, we improve the $\rL^2$-convergence rate from \autoref{lem:algebraic convergence rate L^2} to an $\rH^1$-convergence rate. Note that, it suffices to control the convergence rate of $\rL^2$-norm of the gradient or the dissipation of the system. 

\smallskip

\begin{lem}[Polynomial convergence rate of the gradients]\label{prop:algebraic convergence rate small}
	Let $(\hvair,\hvocn) \in \rC([0,
	+\infty);\rH_\gamma)$ to \eqref{eq:CAOc} be the global strong solution with initial data $(\hvair_0,\hvocn_0)^\top \in \rH_\gamma$. 
	Then the dissipation satisfies
	\begin{equation*}
		\| \nabla^a \hvair(t) \|_{\rL^2(\Omegaair)}^2
		+ 
		\| \nabla \hvocn(t) \|_{\rL^2(\Omegaocn)}^2 
		+ p_s \cdot \| [\hat{v}](t) \|_{\rL^3(\T^2)}^3
		\leq C \cdot (1+t)^{-\frac{3}{2}} \qquad \text{ for all } t \geq 0  
	\end{equation*} 
	for a constant $C > 0$.
\end{lem}
\begin{proof}
	We consider the dissipation of the system
	\begin{equation*}
		D(t) := \int_{\Omegaatm} (|\nablaH \vair|^2 
		+ |p \partial_p \vair|^2) + \int_{\Omegaocn} |\nabla \vocn|^2 +
		\frac{p_s}{3} \int_{\T^2} |[v]|^3 .
	\end{equation*}
	By \autoref{prop:trapping} we may choose $t_0 \in [0,+\infty)$ such that
	\begin{equation*}
		D(t)
		\leq 
		\| \hvair \|_{\rH^1(\Omegaair)}^2 
		+ 
		\| \hvocn \|_{\rH^1(\Omegaocn)}^2
		+ \| [\hat{v}] \|_{\rL^3(\T^2)}^3
		< \epsilon \ll 1 \quad \text{ for all } t \geq t_0 .
	\end{equation*} 
	Now \autoref{lem:differential inequality dissipation} and an absorption argument yields
	\begin{equation*}
		\partial_t D(t)
		+ (\| \Delta^a \hvair \|_{\rL^2(\Omegaair)}^2
		+ \| \Delta \hvocn \|_{\rL^2(\Omegaocn)}^2)
		\leq C \cdot (\| \hvair \|_{\rL^2(\Omegaair)}^2
		+ \| \hvocn \|_{\rL^2(\Omegaocn)}^2) 
		\leq C \cdot (1+t)^{-2} ,
	\end{equation*}
	where the second inequality holds by \autoref{lem:algebraic convergence rate L^2}.
	On the other hand, \autoref{lem:elliptic estimate 2} yields
	the differential inequality
	\begin{equation}
		\partial_t D(t)
		+ c_0 \cdot D(t)^{\frac{4}{3}} \leq C \cdot (1+t)^{-2} . 
		\label{eq:D}
	\end{equation}
	We use the upper barrier function
	\begin{equation*}
		U(t) = a (1+t)^{-\frac{3}{2}}
	\end{equation*}
	where the coefficient $a > 0$ satisfies
	\begin{equation*} 
	 	c_0 \cdot a^{\frac{4}{3}}-\tfrac{3}{2} a \geq C \qquad \text{ and } \qquad a (1+t_0)^{-\frac{3}{2}} \geq D(t_0) .
	\end{equation*}
	A direct calculation confirms 
	\begin{equation*}
		U'(t) + c_0 U(t)^{\frac{4}{3}}
		= -\tfrac{3}{2} a (1+t)^{-\frac{5}{2}}
		 + c_0 a^{\frac{4}{3}} (1+t)^{-2}
		\geq \bigl(-\tfrac{3}{2} a 
		+ c_0 a^{\frac{4}{3}} \bigr) (1+t)^{-2}
		\geq C \cdot (1+t)^{-2} ,
	\end{equation*}
	and the comparison principle of ordinary different equations implies
	\begin{equation*}
		D(t) 
		\leq \tilde{C} \cdot (1+t)^{-\frac{3}{2}} .
	\end{equation*}
	Moreover, we obtain from \autoref{prop:linftyl2} for $t \in [0,t_0]$ the uniform bound
	\begin{equation*}
		\| \nabla^a \hvair(t) \|_{\rL^2(\Omegaair)}^2
		+ 
		\| \nabla \hvocn(t) \|_{\rL^2(\Omegaocn)}^2
		+ p_s \cdot \| [\hat{v}(t)] \|_{\rL^3(\T^2)}^3
		\leq \sup_{t \in [0,t_0]} B(t)^{\frac{1}{2}} = C(t_0) .
	\end{equation*}
	Combining both estimates yields the claim.
\end{proof}

Note that the polynomial convergence rate $\mathcal{O}((1+t)^{-{\frac{3}{2}}})$ in \autoref{prop:algebraic convergence rate small} is not optimal compared to the $\mathcal{O}((1+t)^{-3})$ dictated by the coupling conditions. 
It is even worse then the convergence rate on the energy or $\rL^2$-level. This is due to the error from the advection terms and non-local terms. Indeed, the term of order $\mathcal{O}((1+t)^{-2})$ on the right-hand side of \eqref{eq:D} yields to the slower convergence rate. 
Let us point out how better error terms would yield better decays: an error of order $\mathcal{O}((1+t)^{-3})$ would imply a decay rate $\mathcal{O}((1+t)^{-2})$ which matches the $\rL^2$-rate; an error of order $\mathcal{O}((1+t)^{-4})$ would yield a decay rate $\mathcal{O}((1+t)^{-3})$ which matches the decay rate of the $\rL^2$-norms of the gradients in the toy example. At $\mathcal{O}((1+t)^{-4})$ it saturates: faster errors $\mathcal{O}((1+t)^{-s})$ with $s > 4$ do not yield to better decay. 

We bootstrap our previous argument to obtain the optimal decay rate.

\begin{prop}[Optimal polynomial convergence rate of the gradients]\label{prop:final algebraic convergence rate small}
	Let $(\hvair,\hvocn) \in \rC([0,
+\infty);\rH_\gamma)$ to \eqref{eq:CAOc} be the global strong solution with initial data $(\hvair_0,\hvocn_0)^\top \in \rH_\gamma$. 
 Then the dissipation satisfies 
	\begin{equation*}
		\| \nabla^a \hvair(t) \|_{\rL^2(\Omegaair)}^2
		+ 
		\| \nabla \hvocn(t) \|_{\rL^2(\Omegaocn)} ^2
		+ p_s \cdot \| [\hat{v}(t)] \|_{\rL^3(\T^2)}^{3}
		\leq C \cdot (1+t)^{-3} \qquad \text{ for all } t \geq 0  
	\end{equation*} 
	for a constant $C > 0$.
\end{prop}
\begin{proof}
	We bootstrap the argument from \autoref{prop:algebraic convergence rate small}.
	By \autoref{prop:trapping} it suffices to consider
	\begin{equation*}
		D(t)
		\leq 
		\| \hvair \|_{\rH^1(\Omegaair)}^2 
		+ 
		\| \hvocn \|_{\rH^1(\Omegaocn)}^2
		+ \| [\hat{v}] \|_{\rL^3(\T^2)}^3
		< \epsilon \ll 1 \quad \text{ for all } t \geq t_0 .
	\end{equation*} 
	Further, we recall from the proof of \autoref{lem:differential inequality dissipation} the estimates
	\begin{align*}
		\| \hvair \cdot \nablaH \hvair
		+ \hat{\omega} \partial_p \hvair \|_{\rL^2(\Omegaair)}^2 + \| \nablaH \Phi \|_{\rL^2(\T^2)}^2
		\leq 
		C \cdot 
		\| \hvair \|_{\rH^1(\Omegaair)}^2 \cdot
		\| \hvair \|_{\rH^2(\Omegaair)}^2,\\
		\| \hvocn \cdot \nablaH \hvocn
		+ \hat{w} \partial_z \hvocn \|_{\rL^2(\Omegaocn)}^2
		+ \| \nablaH \pi \|_{\rL^2(\T^2)}^2
		\leq  
		C \cdot 
		\| \hvocn \|_{\rH^1(\Omegaocn)}^2 \cdot
		\| \hvocn \|_{\rH^2(\Omegaocn)}^2.
	\end{align*}
	Since, the arguments for both parts are analogous we only show the oceanic part. 
	Using $\| \hvocn \|_{\rH^1(\Omegaocn)}<\varepsilon$ we obtain 
	\begin{equation*}
		\| \hvocn \|_{\rH^1(\Omegaocn)}^2 \cdot
		\| \hvocn \|_{\rH^2(\Omegaocn)}^2
		= 
		\varepsilon^2 \cdot
		\| D^2 \hvocn \|_{\rH^2(\Omegaocn)}^2
		+ 
		\| \nabla \hvocn \|_{\rL^2(\Omegaocn)}^2 \cdot
		\| \hvocn \|_{\rL^2(\Omegaocn)}^2 
		+ 
		\| \hvocn \|_{\rL^2(\Omegaocn)}^4 .
	\end{equation*}
	The first term on the right-hand side will be absorbed, for the second and the third terms 
	we know from \autoref{lem:algebraic convergence rate L^2} and \autoref{prop:algebraic convergence rate small} that
	\begin{equation}
		\| \nabla \hvocn \|_{\rL^2(\Omegaocn)}^2 \cdot
		\| \hvocn \|_{\rL^2(\Omegaocn)}^2 
		+ 
		\| \hvocn \|_{\rL^2(\Omegaocn)}^4
		\leq \frac{C}{(1+t)^{\frac{3}{2}+2}} +  \frac{C}{(1+t)^{4}}
		\leq  \frac{2C}{(1+t)^{\frac{7}{2}}} .
		\label{eq:bootstrap}
	\end{equation}
	Combining it with the analogous bound for the atmospheric part and using 
	a barrier argument with the upper barrier function
	\begin{equation*}
		U(t) = a(1+t)^{-\frac{21}{8}}
	\end{equation*}
	for a carefully chosen $a > 0$ implies as in the proof of \autoref{prop:algebraic convergence rate small} the decay rate of the dissipation
	\begin{equation*}
		\mathrm{D}(t) \leq C \cdot (1+t)^{\frac{21}{8}}
	\end{equation*}
	and in particular
	\begin{equation*}
		\| \nabla \hvocn(t) \|_{\rL^2(\Omegaocn)} 
		\leq C \cdot (1+t)^{\frac{21}{16}} .
	\end{equation*}
	Now, we run the same argument again. 
	Note that $\frac{21}{8} > 2$ and hence in \eqref{eq:bootstrap} the second term dominates and therefore the bootstrap argument terminates in this step already. One obtains that the error is $\mathcal{O}((1+t)^{-4})$. A barrier argument as in the proof of \autoref{prop:algebraic convergence rate small} yields the desired decay.  
\end{proof}

We conclude the convergence rates for the vertical velocities, the pressure and the geopotential.

\begin{cor}[Optimal polynomial convergence rate of the gradients]\label{cor:w,pi algebraic convergence}
	Let $(\hvair,\hvocn) \in \rC([0,
	+\infty);\rH_\gamma)$ to \eqref{eq:CAOc} be the global strong solution with initial data $(\hvair_0,\hvocn_0)^\top \in \rH_\gamma$. 
	Then the vertical velocities, the pressure and the geopotential decay as follows  
	\begin{align*}
	\| \wair(t) \|_{\rL^2(\Omegaair)}
	+
	\| \wocn(t) \|_{\rL^2(\Omegaocn)}
	\leq C \cdot (1+t)^{-{\frac{3}{2}}} , \qquad \text{ for all } t \geq 0   \\
	\| \nablaH \Phi(t) \|_{\rL^2(\Omegaair)}
	+
	\| \nablaH \pi(t) \|_{\rL^2(\Omegaocn)}
	\leq C \cdot (1+t)^{-2},  \qquad \text{ for all } t \geq 0  
	\end{align*} 
	for a constant $C > 0$.
\end{cor}
\begin{proof}
	The decay rate for the vertical velocities follows immediately from \eqref{eq:wv} and \autoref{prop:final algebraic convergence rate small}.
	The estimates for the geopotential and the oceanic pressure are analogous. We therefore only establish the latter one. We recall from \eqref{eq:pressure} that
	\begin{equation*}
		\nablaH \pi 
		= (\nablaH (-\DeltaH)^{-1} \divH) ( \overline{\vocn \cdot \nablaH \vocn + w^{\ocn} \partial_z \vocn } ) + p_s (\nablaH (-\DeltaH)^{-1} \divH) ( [v] | [v] | )
		=: \nablaH \pi_{\mathrm{NL}} + \nablaH \pi_{\mathrm{CL}} .
	\end{equation*}
	Using \autoref{lem:algebraic convergence rate L^2} and \autoref{prop:final algebraic convergence rate small} and part 1 we obtain by Calderon-Zygmund estimates 
	\begin{equation*}
		\| \nablaH \pi_{\mathrm{NL}} \|_{\rL^2(\Omegaocn)}
		\leq C \cdot (1+t)^{-\frac{5}{2}} .
	\end{equation*}
	For the coupling term we see using Calderon-Zygmund estimates and the embedding
	\begin{equation*}
		\rH^{\frac{1}{2}}(\T^2) \hookrightarrow \rL^4(\T^2)
	\end{equation*} 
	and the trace inequality, that
	\begin{equation*}
		\| \nablaH \pi_{\mathrm{CL}} \|_{\rL^2(\T^2)}
		\leq 
		C \cdot \| [\hat{v}] \|_{\rL^4(\T^2)}^4 
		\leq C \cdot \| [\hat{v}] \|_{\rH^{\frac{1}{2}}(\T^2)}^2  
		\leq C \cdot ( \| \hvair \|_{\rH^1(\Omegaair)}^2 + \| \hvocn \|_{\rH^1(\Omegaocn)}^2 ) 
		\leq C \cdot (1+t)^{-2}. \qedhere 
	\end{equation*}
\end{proof}

The fact that the bootstrap mechanism saturates in finitely many steps strongly suggest that the decay rate from \autoref{prop:final algebraic convergence rate small} is optimal. Moreover, the polynomial rates in \autoref{lem:algebraic convergence rate L^2} and \autoref{prop:final algebraic convergence rate small} coincides with those of the 1D toy model from the introduction. Nevertheless, one may hope for a hidden mechanism which imply faster decay rates than the 1D toy model. 
This cannot happen. 
Indeed, for certain initial data the CAO-system reduces to (a slightly modified version of) the 1D toy model from the introduction. This confirms optimality of the proven decay rates. 

\begin{exa}[Modified toy model]\label{exa:algebraic decay}
	We consider initial data $\vair_0 (x,y,p) = a_0^{\air}(p)$ and $\vocn_0(x,y,z) = a_0^{\ocn}(z)$ for functions $a_0^{\air} \in \rH^1(p_s,p_a)$ and $a_0^{\ocn} \in \rH^1(-h,0)$. Using the uniqueness of the solutions of \eqref{eq:CAO} direct calculations  imply
	\begin{equation*}
		(\vair(t,x,y,p),\vocn(t,x,y,z)) = (a^{\air}(p),a^{\ocn}(z))
	\end{equation*}
	where $a^{\air}$ and $a^{\ocn}$ solve the 1D heat equations
	\begin{equation*}
		\left\{
		\begin{aligned} 
			\partial_t a^\air - \partial_p (p^2 \partial_p a^{\air}) &= 0 
			&\text{ on } (p_s,p_a), \\
			\partial_t a^\ocn -\partial_z^2 a^{\ocn} &= 0 
			&\text{ on } (-h,0) ,
		\end{aligned} 
		\right.
	\end{equation*}
	and the nonlinear coupling conditions
	\begin{align*}
		\begin{aligned} 
			(\partial_p a^{\air})(p_a) = 0 \quad &\text{ and } \quad (\partial_p a^{\air})(p_s) = -p_s^{-1} \cdot |a^{\air}(p_s) -a^{\ocn}(0)|\cdot(a^{\air}(p_s) -a^{\ocn}(0) ) , \\ 
			(\partial_z a^\ocn)(0) = 0 \quad &\text{ and } \quad (\partial_z \vocn)(0) = p_s \cdot |a^{\air}(p_s) -a^{\ocn}(0)|\cdot(a^{\air}(p_s) -a^{\ocn}(0) ) .
		\end{aligned} 
	\end{align*}
\end{exa}

Note that the constant in \autoref{prop:final algebraic convergence rate small} highly depends on the initial data $(\vair_0,\vocn_0)$. 
As the following example shows that there are certain initial data where one obtains exponential decay. This happens because of the absence of the non-linear coupling conditions.

\begin{exa}[Exponential decay for certain initial data]\label{exa:exponential decay}
	We consider initial data $\vair_0 (x,y,p) = f_0(x,y)$ and $\vocn_0(x,y,z) = f_0(x,y)$ for a solenoid vector-field $f \in \rH^1_{\sigma}(\T^2)$. Using the uniqueness of the solutions of \eqref{eq:CAO} we see that 
	\begin{equation*}
		(\vair(t,x,y,p),\vocn(t,x,y,z)) = (f(x,y),f(x,y))
	\end{equation*}
	and $f$ solves the incompressible 2D Navier-Stokes equations
	\begin{equation*}
		\left\{
		\begin{aligned} 
		\partial_t f - \Delta f + f \cdot \nablaH f + \nabla \pi &= 0 
		&\text{ on } \T^2, \\
		\divH f &= 0 &\text{ on } \T^2 .
		\end{aligned} 
		\right.
	\end{equation*}
	It follows (from e.g. a similar argument as in \cite{HK:16,Bin:25} in conjunction with the generalized principle of linearized stability \cite{PSZ:09,PS:16} and the fact that $\sigma(\Delta_{\T^2}) = \{ -|k|^2 \colon k \in \N_0^2 \}$) that
	\begin{equation*}
		\| f(t) - c_* \|_{\rH^1(\T^2)} \leq C e^{-t} \to 0 
	\end{equation*}
	and therefore
	\begin{equation*}
		\| \vair(t) - c_* \|_{\rH^1(\Omegaair)}
		+ 
		\| \vocn(t) - c_* \|_{\rH^1(\Omegaocn)}
		\leq C e^{-t}, 
	\end{equation*}
	where $c_*$ is the common mean 
	\begin{equation*}
		c_* = \frac{1}{|\T^2|} \int_{\T^2} b = \frac{1}{|\Omegaair|} \int_{\Omegaair} \vair
		= \frac{1}{|\Omegaocn|} \int_{\Omegaocn} \vair .
	\end{equation*}
\end{exa}

\bigskip 

\subsection*{Acknowledgement}

The author wants to thank Karoline Disser and Mathias Wilke for helpful discussions about the Lojasiewicz-Simon inequality. 

\subsection*{Funding}

The author would like to thank DFG for support through the project "Globale Existenz und Singularitäten geophysikalischer Flüsse", project no. 538212014.

\bigskip


\end{document}